\documentclass[final]{siamltex}  
\usepackage[latin1]{inputenc}
\usepackage[english,swedish,german]{babel}
\usepackage{algorithmic}
\usepackage{algorithm}
\usepackage{url}

\usepackage{graphicx}
\usepackage{psfrag}
\usepackage[leqno]{amsmath}
\usepackage{tabularx}
\usepackage{verbatim}
\usepackage{float} 
\usepackage{amssymb}
\usepackage{latexsym}
\usepackage{pifont}
\usepackage[loose]{subfigure} 

\usepackage{color}

\newcommand{\sign}{{\operatorname {sign}}}

\newcommand{\re}{{{\mathrm{Re}}~}}
\newcommand{\im}{{{\mathrm{Im}}~}}

\newcommand{\veps}{\varepsilon}

\newcommand{\pder}[2]{\frac{{\partial}#1}{{\partial}#2}}

\newcommand{\RR}{\mathbb{R}}
\newcommand{\DD}{\mathbb{D}}

\newcommand{\CC}{\mathbb{C}}

\newcommand{\kron}{\otimes}

\newtheorem{remark}[theorem]{Remark}

\usepackage{tikz}
\usetikzlibrary{shapes,snakes}

\usepackage{tikz}
\usepackage{pgf}
\usepackage{pgfplots}
\usetikzlibrary{patterns}
\usepgflibrary[shapes.geometric]
\usetikzlibrary{snakes}
\usetikzlibrary{plotmarks}
\usetikzlibrary{patterns}
\usepgflibrary{plothandlers}

\begin{document}
\title{
An inverse iteration method
for eigenvalue problems with 
eigenvector nonlinearities 
}
\author{
Elias Jarlebring\thanks{Dept. of Mathematics, 
NA group, 
KTH Royal Institute of Technology, 
100 44 Stockholm, Sweden.}%
\and%
Simen Kvaal\thanks{Centre for Theoretical and Computational Chemistry, Department of Chemistry, University of Oslo, P.O. Box 1033 Blindern, N-0315 Oslo, Norway}
\and
Wim Michiels\thanks{K.U. Leuven, Department of Computer Science, 3001
    Heverlee, Belgium}
%
}

\selectlanguage{english}
\maketitle
\begin{abstract}
Consider a symmetric matrix $A(v)\in\RR^{n\times n}$ depending on a vector $v\in\RR^n$ and satisfying the property $A(\alpha v)=A(v)$ for any $\alpha\in\RR\backslash\{0\}$.
We will here study the problem of 
finding $(\lambda,v)\in\RR\times \RR^n\backslash\{0\}$ such that $(\lambda,v)$ is an eigenpair of the matrix $A(v)$
and we propose a generalization of inverse iteration 
for eigenvalue problems with this type of eigenvector nonlinearity.
The convergence of the proposed method is studied and 
several convergence properties are shown to be
analogous to inverse iteration for standard eigenvalue problems,
including
local convergence properties.
The algorithm
is also shown to be equivalent to a particular 
discretization of an associated ordinary differential equation,
if the shift is chosen in a particular way.
The algorithm is adapted to 
a variant of the Schr\"odinger equation 
known as the Gross-Pitaevskii equation.
We use numerical simulations to 
 illustrate the convergence properties,
as well as the efficiency of the algorithm
and the adaption. 


\end{abstract}
\section{Introduction}\label{sect:intro}
Let $A:\RR^n\rightarrow\RR^{n\times n}$ be a symmetric matrix
depending on a vector.
We
will consider 
the corresponding  nonlinear eigenvalue problem 
where the vector-valued parameter of $A$, denoted $v\in\RR^n$, equals an eigenvector of the symmetric matrix $A(v)$. That is, 
we wish to find $(\lambda,v)\in\RR\times\RR^n\backslash\{0\}$ such that 
\begin{equation}
 A(v)v=\lambda v,\label{eq:nepv}
\end{equation}
and we will call $\lambda$ an eigenvalue, $v$ an eigenvector 
and $(\lambda,v)$ an eigenpair of \eqref{eq:nepv}. 
In this work we will 
assume that $A$ is three times continuously differentiable 
with respect to $v$ for any $v\neq 0$.

Moreover, we shall  assume that $A$ satisfies 
\begin{equation}
A(\alpha v)=A(v),\;\;\textrm{ for any }\alpha\in\RR\backslash\{0\},\label{eq:invariance}
\end{equation}
such that  the solution is independent of the scaling of  $v$, i.e., 
if $(\lambda,v)$ is a solution
to \eqref{eq:nepv}, then $(\lambda,\alpha v)$ is also a solution for any $\alpha\in\RR\backslash\{0\}$. 
If a problem does not have property \eqref{eq:invariance},
but instead satisfies a normalization constraint,
we can often  transform it to an equivalent problem with property \eqref{eq:invariance}. 
For instance, consider  $\tilde{A}(\cdot)$ 
which is symmetric with 
respect to $v$, i.e.,  $\tilde{A}(v)=\tilde{A}(-v)$.
Also suppose  $\tilde{A}$ does not satisfy \eqref{eq:invariance}.
The problem to find a \emph{normalized} vector
$x\in\RR^n$, i.e., $\|x\|=1$, such that 
  $\tilde{A}(x)x=\lambda x$ 
is equivalent to  \eqref{eq:nepv} 
if we define  
$A(v):=\tilde{A}(v/\|v\|)$ such that $A$  does satisfy \eqref{eq:invariance}.

In this paper we propose a new 
algorithm for \eqref{eq:nepv}, which is a  generalization
of inverse iteration for standard linear eigenvalue problems. 
More precisely, we consider the iteration, 
\begin{equation}
  v_{k+1}=\alpha_k (J(v_k)-\sigma I)^{-1}v_k,\label{eq:invit00}
\end{equation}
where $\alpha_k=1/\|(J(v_k)-\sigma I)^{-1}v_k\|_2$  and 
$J$ is the Jacobian of the left-hand side of \eqref{eq:nepv} with respect to $v$, i.e., 
\begin{equation}
J(v):=\frac{\partial}{\partial v} (A(v)v)\in\RR^{n\times n}.
\label{eq:Jdef}
\end{equation}
The scalar $\sigma\in\RR$ is called the shift and can be 
used to control
to which eigenvalue
the iteration converges,
and has great influence on the speed of convergence.

%

The iteration \eqref{eq:invit00}  is a generalization of inverse
iteration, which is one of the most used 
algorithms for eigenvalue computations. 
Convergence analysis can be found in
\cite{Peters:1979:INVERSE,Ipsen:1997:INVERSEITER} and 
in the studies of Rayleigh quotient iteration in the classical
series of works of Ostrowski, e.g.,  \cite{Ostrowski:1959:RAYLEIGH1}.
 There are also inexact versions
of inverse iteration and its variant Rayleigh quotient iteration, 
e.g., \cite{Simoncini:2002:INEXACT,Freitag:2007:CONVERGENCE}.
These results are often used in combination with
preconditioning techniques for inverse iteration \cite{Neymeyr:2006:GEOMETRIC,Neymeyr:2005:NOTE,Knyazev:2009:GRADIENTFLOW}.
Inverse iteration 
has also been generalized to eigenvalue
problems with eigenvalue nonlinearities, e.g., 
\cite{Neumaier:1985:RESINV,Mehrmann:2004:NLEVP,Voss:2013:NEPCHAPTER} 
for which convergence has been studied  in \cite{Jarlebring:2012:CONVFACT,Jarlebring:2011:RESINVCONV}. 
There is  an 
algorithm for problems with 
eigenvector nonlinearities and
based on solving linear systems in \cite{Meyer:1997:NONLINEAR}.
See also 
the results on eigenvector nonlinearities in 
\cite{Demoulin:1975:ITERATION}.
To our knowledge, the 
iteration \eqref{eq:invit00} for problems
of the type \eqref{eq:nepv} has not been presented
or analyzed 
in  the literature.  

\begin{table}[t]
\begin{center}
\begin{tabular}{l|c|c|l}
           &  Shift $\sigma$   & Error $\|v_k-v_*\|$ & Characterization\\ \hline
Section~\ref{sect:localconv} &  arbitrary   & $\|v_{k}-v_*\|\ll1$ & One step of  \eqref{eq:invit00} satisfies the  \\ &&&first-order expansion 
\eqref{eq:intro_localerr}\\
\hline
Section~\ref{sect:flow} &  $\sigma\ll \lambda_*$ & arbitrary & One step of \eqref{eq:invit00} approximates the \\&&&
 trajectory of the ODE \eqref{eq:intro_flow}\\
\end{tabular}
\caption{Applicability of the convergence-characterizations in this
paper\label{tbl:applicability}}
\end{center}
\end{table}

We characterize the convergence of the
iteration \eqref{eq:invit00} in two ways,
which lead to conclusions for the behavior 
in two situations; when the
error $\|v_k-v_*\|$ is small or
when the shift satisfies
 $\sigma\ll\lambda_*$. See Table~\ref{tbl:applicability}. 

In particular, in the local convergence analysis (in Section~\ref{sect:localconv})
we show the following. 
 For any eigenpair $(\lambda_*,v_*)$ of \eqref{eq:nepv}, the iteration satisfies 
\begin{equation}
v_{k+1}\pm v_*=|\lambda_*-\sigma|F_*(v_k-v_*)+O(\|v_k-v_*\|^2),\label{eq:intro_localerr},
\end{equation}
where the sign depends on $\sign(\lambda_*-\sigma)$,
and we derive an expression for $F_*\in\RR^{n\times n}$.
If the shift is sufficiently close
to the eigenvalue, the iteration is locally convergent.
 The convergence is in general linear and the convergence
 factor is proportional to the
distance between the shift and the eigenvalue.

We also provide a characterization
of \eqref{eq:invit00} by deriving a relation with an ODE
applicable in the situation
when $\sigma\ll\lambda_*$. 
In particular, in Section~\ref{sect:flow} we show 
 the following. One step of \eqref{eq:invit00}
is equivalent to a particular type of discretization of the ODE 
\begin{equation}
y'(t)=p(y(t))y(t)-A(y(t))y(t),\label{eq:intro_flow}
\end{equation}
where $p$ is the Rayleigh quotient 
\begin{equation}
p(y):=\frac{y^TA(y)y}{y^Ty},\label{eq:rq}
\end{equation}
if the shift is chosen such that  $\sigma\ll\lambda_*$.
The stationary solutions of \eqref{eq:intro_flow} are solutions to \eqref{eq:nepv}.
 A small step-length in the discretization corresponds to
$\sigma\ll \lambda_*$, implying that one
step of \eqref{eq:invit00} approximates 
the trajectory of \eqref{eq:intro_flow} if $\sigma$ is chosen
sufficiently negative. 
In the linear case (when $A(v)=B$) when the matrix $B$ has
one simple dominant eigenvalue (eigenvalue closest to $-\infty$), 
the ODE only has one 
stable stationary point (which corresponds to the dominant eigenvalue)
 and it is convergent for any starting
value. If a similar situation occurs for the nonlinear case, i.e.,
 the ODE \eqref{eq:intro_flow}
is convergent and only has one
stable stationary point, then  
the iteration \eqref{eq:invit00} is
expected to converge to this solution for sufficiently negative $\sigma$
and the convergence is expected to be independent of starting vector.
In particular, if the problem is close to linear, the iteration 
is expected to  converge to  the dominant eigenvalue.


%
%
The idea we use in 
Section~\ref{sect:flow}, basing the  
reasoning on interpreting the
iterative method as a discretization or
realization of an ODE, has been used in a
number of other settings for
eigenvalue problems before, e.g.,
for the QR-method
\cite{Watkins:1984:ISOSPECTRAL,Watkins:1988:SELFSIMILAR,Chu:1988:CONTINUOUS},
preconditioning techniques
\cite{Knyazev:2009:GRADIENTFLOW} 
and characterizations of the Rayleigh
quotient iteration \cite{Mahony:2003:RAYLEIGHFLOW}.
In fact, the ODE \eqref{eq:flow0}
is a nonlinear variant of the Oja flow \cite{Yan:1994:OJA}.
See also the collection \cite{Bloch:1994:BLOCH} and
the description of
iterations on manifolds  in 
\cite{Edelman:1998:GEOMETRY}.

The  algorithm is illustrated by 
showing how it can be adapted to the Gross-Pitaevskii equation (GPE)
in Section~\ref{sect:GP}. 
The GPE is a standard model
for particles in the state of matter called the
Bose-Einstein condensate. See \cite{Landau:1965:QM} and 
references in \cite{Bao:2004:BOSEEINSTEIN} for literature
on the GPE.
We discretize the GPE and transform it to  the form \eqref{eq:nepv}.
We also show how the solution to 
the linear system $(J(v_k)-\sigma I)^{-1}v_k$ can be found 
efficiently using the Sherman-Morrison-Woodbury formula. 
It turns out that in this setting, the ODE \eqref{eq:intro_flow}
 is  
directly 
related to the  technique called 
 \emph{imaginary time-integration} \cite{Bao:2004:BOSEEINSTEIN}.
This connection allows us to use results
known for imaginary time-integration, in particular that the ODE
 always converges to  a stationary solution and the
experience presented in results in the literature
indicate that this is often a physically relevant
 solution, e.g., the ground state.
We further study the ODE and
derive a heuristic
choice of the step-length $h$, or equivalently a
heuristic choice for the shift $\sigma$, which follows
the trajectory of the ODE to sufficient accuracy,
and still maintains a fast asymptotic convergence rate.

Although we are not aware of algorithms
for \eqref{eq:nepv}, a number of
successful  algorithms
do exist for the specific
application we have in mind, i.e., the GPE in Section~\ref{sect:GP}. 
Besides the aforementioned methods based on 
imaginary time-integration, 
there are methods for the GPE based on 
minimization  \cite{Bao:2003:GROUND,Caliari:2009:MINIMISATION}
as well as a
Newton-Rhapson approach \cite{Choi:2002:NLEVPV}.
 Our approach is different
in character since the reasoning
stems from an 
eigenvalue algorithm (inverse iteration), 
and it has a different generality
setting. It can be interpreted
as a discretization of the ODE, but
with an integration scheme and step-length 
which 
we believe has not been used for the
imaginary time-integration of the GPE. 
For completeness we also present results for
another generalization of the algorithm
which turns out to be essentially 
 equivalent  to 
an approach presented in \cite{Bao:2004:BOSEEINSTEIN}
  and
can be seen as \eqref{eq:invit00} 
where $J$ is replaced by $A$.

The notation in this paper is mostly standard. 
We use $\frac{\partial}{\partial v}q(v)$ 
to denote the Jacobian of a vector or scalar $q$,
i.e., if $v\in\RR^n$ and $q(v)\in\RR^k$,  then 
$\frac{\partial}{\partial v}q(v)= 
(\frac{\partial}{\partial v_1}q(v),\ldots,\frac{\partial}{\partial v_n}q(v))\in\RR^{k\times n}$.
 We use the notation $(\cdot)_{x=y}$ to denote substitution of $x$ with $y$ 
for the formula inside the parenthesis. 
As usual, the expression $\|Z\|_2$ denotes
the Euclidean norm if $Z$ is a vector and the spectral norm
 if $Z$ is a matrix.  
The set of  eigenvalues of a matrix $B\in\RR^{n\times n}$
will be denoted $\lambda(B)$ and a dominant
eigenvalue will be used to refer to an 
eigenvalue $\mu\in\lambda(B)$ for
which all other eigenvalues have larger (or equal) real part, i.e.,
if $\mu\in\lambda(B)$ is a dominant
eigenvalue then, any $\mu_2\in\lambda(B)$ satisfies $\re(\mu)\le \re(\mu_2)$. 

\section{Preliminaries and fixed point formulation}
Throughout this paper
we will in several situations use the following
consequences of
the scaling invariance property \eqref{eq:invariance}.
\begin{lemma}[Scaling invariance]\label{thm:scalinv}
Consider a scaling invariant function $Q:\RR^n\rightarrow\RR^{k\times
  \ell}$, i.e., $Q(\alpha v)=Q(v)$ for any
$\alpha\in\RR\backslash\{0\}$. Then, for any vectors  $u\in \RR^{\ell}$, $v\in\RR^n$,
\begin{equation}
 \left(\frac{\partial}{\partial v} (Q(v)u)\right)v=0.\label{eq:invariance1}
\end{equation}
In particular, given $A:\RR^n\rightarrow\RR^{n\times n}$
satisfying the scaling invariance \eqref{eq:invariance} 
and $J$ defined by \eqref{eq:Jdef}, we have
\begin{equation}
  J(u)u=A(u)u\label{eq:JAidentity}
\end{equation}
for any $u\in\RR^n$. Moreover, if $u=v_*$, when $(\lambda_*,v_*)$
is an eigenpair of \eqref{eq:nepv}, then $\lambda_*$ is
an eigenvalue of $J(v_*)$ with eigenvector $v_*$.
\end{lemma}
\begin{proof}
The left-hand side of \eqref{eq:invariance1} can
be interpreted as a directional
in the direction of $v$. We have 
\[
\left(\frac{\partial}{\partial v} (Q(v)u)\right)v:=
\lim_{\veps\rightarrow 0}\frac{Q(v+\veps v)u-Q(v)u}{\veps}=
\lim_{\veps\rightarrow 0}\frac{Q((1+\veps)v)u-Q(v)u}{\veps}=0,
\]
which shows \eqref{eq:invariance1}.
The identity \eqref{eq:JAidentity}
follows from the chain rule.
\end{proof}

The iteration can naturally be represented in a 
fixed point form, 
\[
v_{k+1}=\varphi(v_k)
\]
where
\begin{equation}
  \varphi(v):=\frac{1}{\|\psi(v)\|}\psi(v)\label{eq:phidef}
\end{equation}
and
\begin{equation}
 \psi(v):= (J(v)-\sigma I)^{-1}v.\label{eq:psidef}
\end{equation}
Obviously, $v_{k+1}=\varphi(v_k)$ when $v_k$
is given by \eqref{eq:invit00}. 
It also turns out that the fixed points of $\varphi$ are
solutions to \eqref{eq:nepv}. However, the
converse is not true. We shall 
now show that 
if $\sigma>\lambda_*$, we have $\varphi(v_*)=-v_*$,
implying that the iterates $v_k$ alternate between
$v_*$ and $-v_*$ if $v_0=v_*$.
An equivalence with the solutions to \eqref{eq:nepv}
and the vectors that satisfy $\varphi(v_*)=\pm v_*$,
 for some choice of the sign,
can be achieved if 
we take the alternation into account,
as can be seen as follows. 

\begin{proposition}[Equivalence of fixed points and eigenvectors] \label{thm:fixedp}
\begin{itemize}
\item[(a)] 
Suppose $(\lambda_*,v_*)$  is a solution to \eqref{eq:nepv} 
with $\|v_*\|=1$ 
and suppose $\sigma\neq\lambda_*$. Then,
the fixed point map $\varphi$ satisfies 
\begin{equation}
   \varphi(v_*)=\pm v_*.\label{eq:isafp}
\end{equation}
The sign in \eqref{eq:isafp} should be chosen positive if $\sigma<\lambda_*$,
otherwise negative.
\item[(b)] Suppose $v_*=\varphi(v_*)$   or 
$v_*=-\varphi(v_*)$. Then, 
$(\lambda_*,v_*)$ is
a solution to  \eqref{eq:nepv}
with 
\begin{equation}
\lambda_*:=\sigma\pm \frac{1}{\|(J(v_*)-\sigma I)^{-1}v_*\|}=v_*^TA(v_*)v_*.\label{eq:fp_lambda}
\end{equation}
\end{itemize}
%
\end{proposition}
\begin{proof}
In order to show (a), suppose 
$(\lambda_*,v_*)$ is an eigenpair and note that 
from Lemma~\ref{thm:scalinv} we have that 
\[
  \lambda_* v_*=A(v_*)v_*=J(v_*)v_*.
\]
By subtracting $\sigma v_*$ from both sides and subsequently
multiplying both sides 
from the left by $\frac{1}{\lambda_*-\sigma}(J(v_*)-\sigma I)^{-1}$ 
we can simplify
\[
  \psi(v_*)= 
(J(v_*)-\sigma I )^{-1}v_*=
\frac{1}{\lambda_*-\sigma}v_*,
\]
and we consequently  find that 
\begin{equation}
  \frac{1}{\|\psi(v_*)\|}=\frac{1}{\sqrt{\psi(v_*)^T\psi(v_*)}}=|\lambda_*-\sigma|.
\label{eq:sqrtlambda}
\end{equation}
Hence, $\varphi$  evaluated at $v=v_*$ is explicitly 
\[
  \varphi(v_*)=
  \frac{1}{\sqrt{\psi(v_*)^T\psi(v_*)}}\psi(v_*)=
\frac{|\lambda_*-\sigma|}{\lambda_*-\sigma} v_*=\pm v_*.
\]
In order to show (b), suppose that $v_*$ is
such that  $\varphi(v_*)=\pm v_*$. 
Equation \eqref{eq:phidef} leads to 
\[
   (J(v_*)-\sigma I)v_*=\frac{\pm1}{\|(J(v_*)-\sigma I)^{-1}\|}v_*.
\]
From Lemma~\ref{thm:scalinv} it follows that  $J(v_*)v_*=A(v_*)v_*$ and 
we finally have that 
\[
  A(v_*)v_*= \left(\sigma\pm \frac{1}{\|(J(v_*)-\sigma I)^{-1}v_*\|}\right)v_*.
\]  
and the formula \eqref{eq:fp_lambda} follows
by multiplying from the left with $v_*^T$. 
\end{proof}

\section{Local convergence properties}~\label{sect:localconv}
\subsection{First-order behavior and convergence factor}
From  Proposition~\ref{thm:fixedp}
we can directly conclude that if 
we at some point in 
the iteration have $v_k=v_*$, where $v_*$
is an eigenvector, then
every subsequent iterate $v_j$, $j>k$,  will 
also be an eigenvector
corresponding to the same eigenvalue, but
can 
possibly alternate between  $\pm v_*$. 
We will now study the case
where $v_k$ is close but not equal to an eigenvector. 
In order to understand
this local convergence behavior, we also need to take  
the alternation into account. 
The error is to first order given by the following result.

\begin{theorem}[Local convergence]\label{thm:localconv}
Suppose $(\lambda_*,v_*)$  is a solution to \eqref{eq:nepv}
with $\|v_*\|=1$. Let $\sigma\in\RR$ be any
shift such that $J(v_*)-\sigma I$ is non-singular, 
in particular $\sigma\neq\lambda_*$.
 Then,
\begin{equation}
  \varphi'(v_*)=|\lambda_*-\sigma|(I-v_*v_*^T)%
\left(J(v_*)-\sigma I\right)^{-1}.\label{eq:varphiprim}
\end{equation}
Moreover, 
suppose the iterates $v_k$ generated by  
\eqref{eq:invit00} are such that
$J(v_k)-\sigma I$ is non-singular for any $k$. Then,
the iterates 
satisfy
\begin{equation}
   v_{k+1}\mp v_*=\varphi'(v_*)(v_k-v_*)+O(\|v_k-v_*\|^2)\label{eq:vkp1vk}
\end{equation}
The sign in \eqref{eq:vkp1vk} should be chosen negative if $\sigma<\lambda_*$,
otherwise positive.
%
\end{theorem}
\begin{proof}
By using the Taylor expansion and 
the fixed point characterization in Proposition~\ref{thm:fixedp},
we have that the left-hand side of \eqref{eq:vkp1vk} 
satisfies
\[
  v_{k+1}\mp v_*=
  v_{k+1}-(\pm v_*)=\varphi(v_k)-\varphi(v_*)=
\varphi'(v_*)(v_k-v_*)+O(\|v_k-v_*\|^2).
\]
It remains to derive the formula 
$\varphi'(v_*)$ given by \eqref{eq:varphiprim}.

We will need the derivative of the two-norm, which can be 
expressed as the row-vector
\begin{equation}
\frac{\partial}{\partial v}\|\psi(v)\|_2= \frac{\partial}{\partial v}\sqrt{\psi(v)^T\psi(v)}=
\frac{1}{\sqrt{\psi(v)^T\psi(v)}}\psi(v)^T\psi'(v)=
\varphi(v)^T\psi'(v),\label{eq:sqrtder}
\end{equation}
since $\varphi(v)^T=\psi(v)^T/\|\psi(v)\|_2$.

From the definition \eqref{eq:phidef} of $\varphi$ 
we have the relation  $\varphi(v)\|\psi(v)\|=\psi(v)$,
whose Jacobian can be computed with the product rule and \eqref{eq:sqrtder}, 
\[
  \varphi'(v)\|\psi(v)\|+ \varphi(v)
\varphi(v)^T
\psi'(v)
=\psi'(v).
\]
Hence, 
\begin{equation}
\varphi'(v)=\frac{1}{\|\psi(v)\|}
\left(I-\varphi(v)\varphi(v)^T \right)\psi'(v).\label{eq:localconv_proof_varphip}
\end{equation}
Now recall 
(from Proposition~\ref{thm:fixedp})  that  $\varphi(v_*)=\pm v_*$  and consequently 
$\varphi(v_*)\varphi(v_*)^T=v_*v_*^T$. This
fact combined with the formula for the norm
\eqref{eq:sqrtlambda} leads to a
simplification of \eqref{eq:localconv_proof_varphip} when 
we evaluate at  $v=v_*$, 
\begin{equation}
\varphi'(v_*)=|\lambda_*-\sigma|
\left(I-v_*v_*^T \right)\psi'(v_*).\label{eq:phivs}
\end{equation}
It remains to establish a formula for the Jacobian of $\psi$ 
evaluated at $v=v_*$. 
By differentiation of \eqref{eq:psidef} multiplied by $J(v)-\sigma I$, we have
%
\begin{equation}
\left(\frac{\partial}{\partial v}
 (J(v)-\sigma I)\psi(\hat{v})\right)_{\hat{v}=v}+
 (J(v)-\sigma I)\psi'(v)
= I\label{eq:Jpsider}
\end{equation}
and
\begin{equation}
\psi'(v_*)=
-(J(v_*)-\sigma I)^{-1}
\left(\frac{\partial}{\partial v}
 J(v)\psi(v_*)\right)_{v=v_*}+
 (J(v_*)-\sigma I)^{-1}.\label{eq:psider}
\end{equation}
Moreover, note that $\psi(v_*)=\|\psi(v_*)\|\varphi(v_*)=\pm \|(J(v_*)-\sigma I)^{-1}v_*\|v_*$.
We will now show that the first term in \eqref{eq:psider} vanishes
identically by showing that
all columns of $\left(\frac{\partial}{\partial v}J(v)v_*\right)_{v=v_*}$ vanish.
Let the $j$th column be denoted 
\begin{equation}
c_j:=
\left(\frac{\partial}{\partial v}J(v)v_*\right)_{v=v_*}e_j=
\lim_{\veps\rightarrow 0}\frac{1}{\veps}
\left(J(v_*+\veps e_j)v_*-J(v_*)v_*\right).\label{eq:rkdef}
\end{equation}
Lemma~\ref{thm:scalinv} and in particular the
 identity  $J(u)u=A(u)u$
for any $u\in\RR^n$, implies that
\begin{multline*}
 J(v_*+\veps e_j)v_*-J(v_*)v_*=\\
J(v_*+\veps e_j)(v_*+\veps e_j)-\veps J(v_*+\veps e_j)e_j-J(v_*)v_*=\\
A(v_*+\veps e_j)(v_*+\veps e_j)-\veps J(v_*+\veps e_j)e_j-J(v_*)v_*=\\
A(v_*)v_*+\veps J(v_*)e_j
-\veps J(v_*+\veps e_j)e_j-J(v_*)v_* +O(\veps^2),
\end{multline*}
where we used a Taylor expansion of $A(v_*+\veps e_j)(v_*+\veps e_j)
=A(v_*)v_*+J(v_*)(\veps e_j)+O(\veps^2)$ in
the last step. 
Hence, by again applying 
 Lemma~\ref{thm:scalinv} now giving
 $A(v_*)v_*=J(v_*)v_*$, 
we have
\[
c_j= \lim_{\veps\rightarrow 0}
\big(J(v_*)e_j-
J(v_*+\veps e_j)e_j+O(\veps)\big)=0.
\]
We have shown that $c_j=0$, for any $j=1,\ldots,n$
which implies that 
\begin{equation}
\left(\frac{\partial}{\partial v}J(v)v_*\right)_{v=v_*}=0\in\RR^{n\times n}
\label{eq:Jis0}
\end{equation}
such that the first term in \eqref{eq:psider} vanishes and
the proof is complete. 
\end{proof}

%
The above theorem for the local behavior 
directly gives a characterization of
the convergence factor in terms of the largest
eigenvalue of $\varphi'(v_*)$ in modulus. 
To formalize this statement we will use the concept
of $Q$-convergence factor and $R$-convergence factor
given in \cite[Chapter~9]{Ortega:2000:NEWTON}.
We briefly summarize the concepts for the generic situation
and  our setting. 
Suppose a sequence $\{v_k\}_{k=0}^{\infty}$ 
converges linearly to $v_*$. Then, 
generically, the $R$-convergence
factor (for the sequence  $\{v_k\}_{k=0}^{\infty}$) is given by 
\[
  R_1(\{v_k\})=\lim_{k\rightarrow\infty}\|v_k-v_*\|^{1/k}.
\]
Correspondingly, the $Q$-convergence factor is given by
\[
 Q_1(\{v_k\})=\lim_{k\rightarrow\infty}\frac{\|v_{k+1}-v_*\|}{ \|v_k-v_*\|}.
\]
The $R$-convergence factor and $Q$-convergence factor
associated with a fixed point iteration and 
a particular fixed point is defined as the 
supremum of the convergence
factor for all non-degenerate convergent 
sequences generated by the fixed point iteration,
and will be denoted 
$R_1(\varphi,v_*)$ and $Q_1(\varphi,v_*)$.
See \cite[Chapter~9]{Ortega:2000:NEWTON} for 
formal definitions.
%
%

When $\sigma<\lambda_*$ we can conclude from Proposition~\ref{thm:fixedp}
that $v_*$ is a fixed point of $\varphi(v_*)$. 
On the other
hand, if $\sigma>\lambda_*$ we need
to take alternation into account. It 
is easy to verify that $v_*$
is fixed point of $\hat{\varphi}(v):=-\varphi(v)$ 
which is a fixed point iteration
giving the same sequence of vectors $v_k$ generated by $\varphi$, except that
every second vector has a different sign. 
This compensates for the 
asymptotic sign alternation, such that 
convergence to an oscillating sequence $v_k$
using $\varphi$ is equivalent to 
a convergence to a fixed point using $\hat{\varphi}$
and allows us to 
study both cases in a unified manner.
We will now see that these definitions
of convergence factors give one
formula for the convergence factor (independent of 
the two cases $\lambda_*<\sigma$ or $\lambda_*>\sigma$) 
involving eigenvalues of $J(v_*)$. Moreover,
the $R$-convergence factors and $Q$-convergence
factors are in the generic situation equal. 
%
%
%
%
%
\begin{corollary}[Convergence factor]\label{thm:convfact}
Let  $(\lambda_*,v_*)$ be an eigenpair of \eqref{eq:nepv}
with $\|v_*\|=1$.
Suppose $\lambda_*$  is a simple eigenvalue of $J(v_*)$ and 
let
\begin{equation}
\gamma=\frac{|\lambda_*-\sigma|}{|\mu_2-\sigma|}\label{eq:convfact}
\end{equation}
where $\mu_2\in\CC$ is the eigenvalue of $J(v_*)$ 
closest to $\sigma$, but not equal to $\lambda_*$, i.e., 
\[
  \mu_2=\underset{{\mu\in\lambda(J(v_*))\backslash\{\lambda_*\}}}{\operatorname{argmin}} |\mu-\sigma|
\]
and suppose it is simple.
Let $\hat{\varphi}$ be defined by 
\begin{itemize}
\item $\hat\varphi:=\varphi$ when $\sigma<\lambda_*$; and
\item $\hat\varphi:=-\varphi$ when $\sigma>\lambda_*$. 
\end{itemize}
Suppose $\gamma<1$. 
Then  the fixed point $v_*$ is locally convergent with 
respect to $\hat\varphi$ 
and the 
convergence is linear (in both $R$ 
order and $Q$ order). Moreover, the
$R$-convergence factor is given by 
\begin{equation}
    R_1(\hat\varphi,v_*)=\gamma.
\end{equation}
Furthermore, assume that  $\{v_k\}_{k=0}^\infty$ 
is a sequence generated by $\hat{\varphi}$
convergent to $v_*$ and such that $R_1(\{v_k\})=\gamma$.
Then, under the condition that  
$\|v_{k+1}-v_*\|/\|v_k-v_*\|$ converges to a nonzero value, 
the $Q$-convergence factor is 
equal to the $R$-convergence factor, i.e., 
\begin{equation}
  Q_1(\{v_k\})=\gamma. \label{eq:Q_factor}
\end{equation}
\end{corollary}
\begin{proof}
Consider
any of the two iterations $v_{k+1}=\varphi(v_k)$ and $v_{k+1}=-\varphi(v_k)$. 
Ostrowski's theorem \cite[Theorem~10.1.3]{Ortega:2000:NEWTON},
 the linear convergence theorem \cite[Theorem~10.1.4]{Ortega:2000:NEWTON} 
and Theorem~\ref{thm:localconv} gives
the characterization 
of the $R$-convergence factor using the spectral radius 
\[
 R_1(\{v_k\})=\rho((I-v_*v_*^T)(J(v_*)-\sigma I)^{-1}),
\]
independent of which of the two iterations are considered. 

Now note that $v_*$ is an eigenvector of $J(v_*)$ 
and also an eigenvector of $(J(v_*)-\sigma I)^{-1}$. 
The application of
the projector $(I-v_*v_*^T)$ has the effect that 
the eigenvalues of $(I-v_*v_*^T)(J(v_*)-\sigma I)^{-1}$
are the same as the  eigenvalues of  $(J(v_*)-\sigma I)^{-1}$
except for the eigenvalue corresponding
to $v_*$ (which by assumption is simple). 
This eigenvalue  is transformed to zero, hence
the maximum must be taken over the 
eigenvalues of $(J(v_*)-\sigma I)^{-1}$ 
except for the $(\sigma-\lambda_*)^{-1}$. 
We shown \eqref{eq:convfact} have the formula for $\gamma$. 

The formula \eqref{eq:Q_factor} follows
from the fact that the linear $Q$-convergence factor is equal to the
linear $R$-convergence factor under the assumption that the
error quotient $\|v_{k+1}-v_*\|/\|v_k-v_*\|$ converges, which holds
 by assumption.
See e.g. \cite[NR~10.1-5]{Ortega:2000:NEWTON}.

\end{proof}

\subsection{Similarities and differences with inverse iteration for linear problems}

The iteration \eqref{eq:invit00}
 is clearly a generalization of
inverse iteration when $A(v)=A_0$ is constant,
and a behavior similar to (linear) inverse
iteration is expected when the problem 
is close to linear. More importantly,
from the theory
above we can conclude that the
iteration possesses many properties similar to 
inverse iteration also in a  situation  when the problem 
is not close to linear. 

\begin{itemize}
\item The convergence is in general linear. 
\item The convergence factor is 
asymptotically proportional to eigenvalue shift distance
when the distance is small.
\item The convergence factor
 \eqref{eq:convfact}
 is in general determined by
a quotient involving  the eigenvalue shift distance in the numerator. 
Unlike the linear case, the denominator is the distance between
the shift and the second closest eigenvalue of $J(v_*)$. 
\item Unlike the linear
case, for a given shift, several eigenvectors
can be attractive fixed points. Conversely, for a given shift, the iteration may not have any 
attraction points. 
\end{itemize}
In order to further illustrate the value of these properties
we will in this work 
also briefly study another generalization 
of inverse iteration. Inverse iteration for standard
eigenvalue problems consists of shifting, inverting 
and normalizing, and it is from this perspective
natural to consider the generalization  $v_{k+1}=\varphi_A(v_k)$,  
\begin{equation}
  \label{eq:phidef-A}
  \varphi_A(v) := \frac{1}{\|\psi_A(v)\|} \psi_A(v), \quad \psi_A(v) = (A(v)-\sigma I)^{-1}v.
\end{equation}
We will call the $A$-version of inverse iteration.

We can derive a description 
of the first-order behavior 
similar to Theorem~\ref{thm:localconv}. 
Note that $\varphi_A'(v_*)$ is
considerably more complicated than $\varphi'(v_*)$ 
and most of the bullets above do not
apply to this version, most importantly,  the convergence factor is not necessarily
small if the shift is close to the eigenvalue. 

%
%
%

\subsection{Illustration  of local convergence}\label{sect:localconv_exmp}
Several properties of the algorithm, including
the local convergence above, 
can be  observed  when applied
to the example 
\[
  A(v)=A_0+\sin(\frac{v^TBv}{v^Tv}) A_1.
\]
The Jacobian is given by
\[
  J(v)=\frac{\partial}{\partial v}A(v)v=A(v)+2\frac{\cos(\frac{v^TBv}{v^Tv})}{(v^Tv)^2}A_1v((v^Tv)v^TB-(v^TBv)v^T).
\]
%
%
We selected $A_0$, $A_1$ and $B$ in a random way. In order 
to make the numerical simulations reproducible, 
we will fix $A_0$, $A_1$ and $B$ 
as follows
\begin{eqnarray*}
A_0&=&
\frac{1}{10}
\begin{pmatrix}
    10 &   21 &   13 &   16\\
    21 &  -26 &   24 &    2\\
    13 &   24 &  -26 &   37\\
    16 &    2 &   37 &   -4
\end{pmatrix},
A_1=
\frac{\beta}{10}
\begin{pmatrix}
    20  &  28 &   12 &   32\\
    28  &   4 &   14 &    6\\
    12  &  14 &   32 &   34\\
    32  &   6 &   34 &   16
\end{pmatrix},\\
B&=&
\frac{1}{10}
\begin{pmatrix}
   -14  &  16 &   -4  &  15\\
    16  &  10 &   15  &  -9\\
    -4  &  15 &   16  &   6\\
    15  &  -9 &    6  &  -6
\end{pmatrix}
\end{eqnarray*}
where we carry out
simulations for a number of different $\beta$. Note that 
the problem is linear when $\beta=0$.

Simulations of the inverse iterations
for different $\beta$ are given in 
Figure~\ref{fig:AvsJerror} and
Figure~\ref{fig:sinexmp:convfact}.
In Figure~\ref{fig:AvsJerror} 
we clearly observe linear convergence,
for the $A$-version 
as well as the $J$-version. The $J$-version 
is also clearly less dependent on $\beta$. 
As  the nonlinearity parameter $\beta$ increases,
the performance of the $A$-version 
worsens whereas the convergence
of the  $J$-version is not greatly affected by increase of $\beta$. 
We see in Figure~\ref{fig:AvsJerror}f that the slow
convergence of the $A$-version
can not always be compensated
by moving $\sigma$ closer 
to the eigenvalue $\lambda_*$.
In this case the $A$-version 
is not 
even locally convergent when $\sigma=\lambda_*$, whereas
the $J$-version exhibits quadratic convergence. Note that 
for $\sigma=\lambda_*$ the fixed points map $\varphi(v_*)$ and
$\varphi_A(v_*)$ involve inverses
of singular matrices, implying that they  
do formally  not have fixed points $v_*$.
is invertible.
 Similar to inverse iteration for standard eigenvalue 
problems the $J$-version still works in practice when subject
to rounding errors.

\begin{figure}[H]
\begin{center}
\subfigure[$\beta=0.01$, $\sigma=\lambda_*+0.3$]{\scalebox{0.6}{\includegraphics{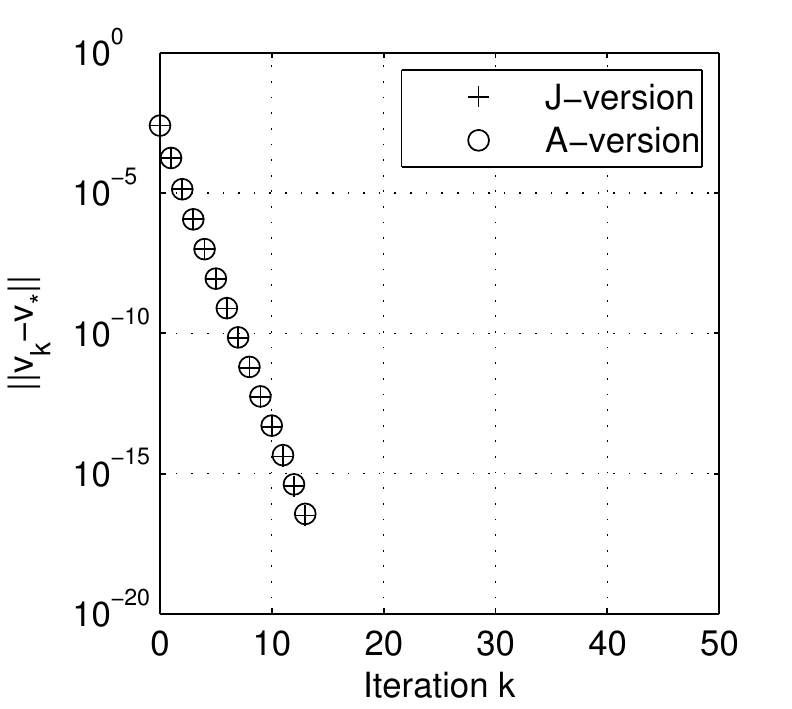}}}%
\subfigure[$\beta=0.1$, $\sigma=\lambda_*+0.3$]{\scalebox{0.6}{\includegraphics{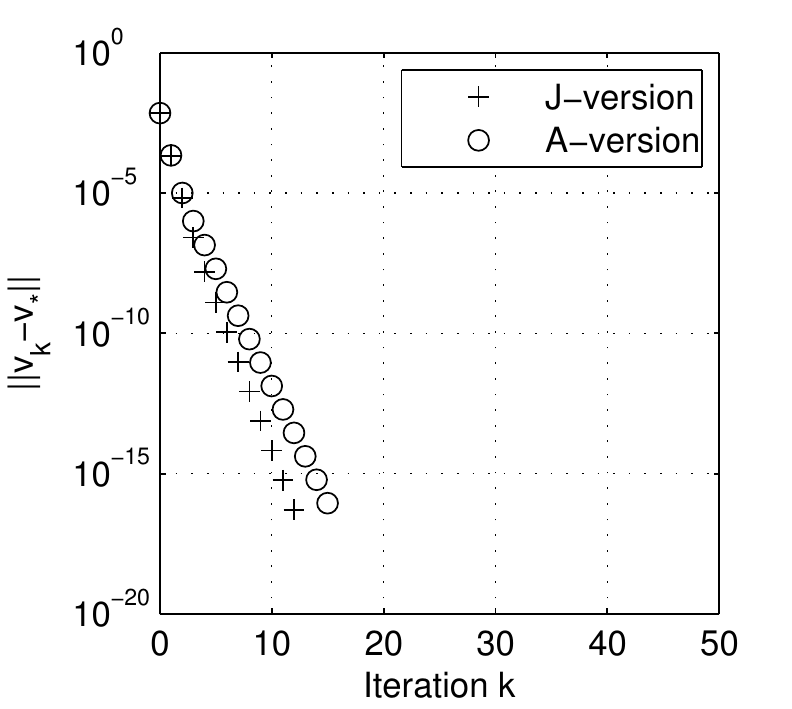}}}
\subfigure[$\beta=0.5$, $\sigma=\lambda_*+0.3$]{\scalebox{0.6}{\includegraphics{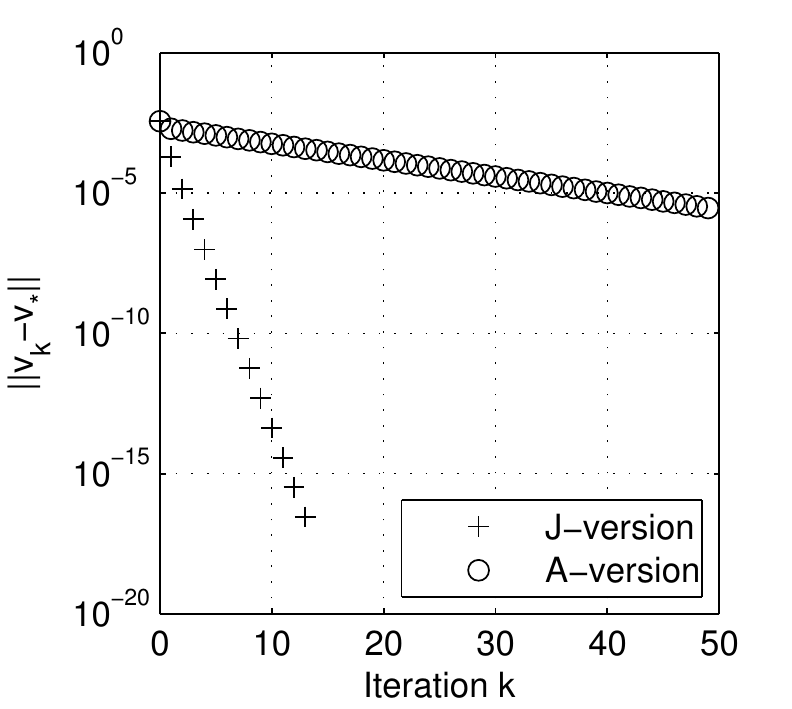}}}\vspace{-0.2cm}%
\subfigure[$\beta=0.6$, $\sigma=\lambda_*+0.3$]{\scalebox{0.6}{\includegraphics{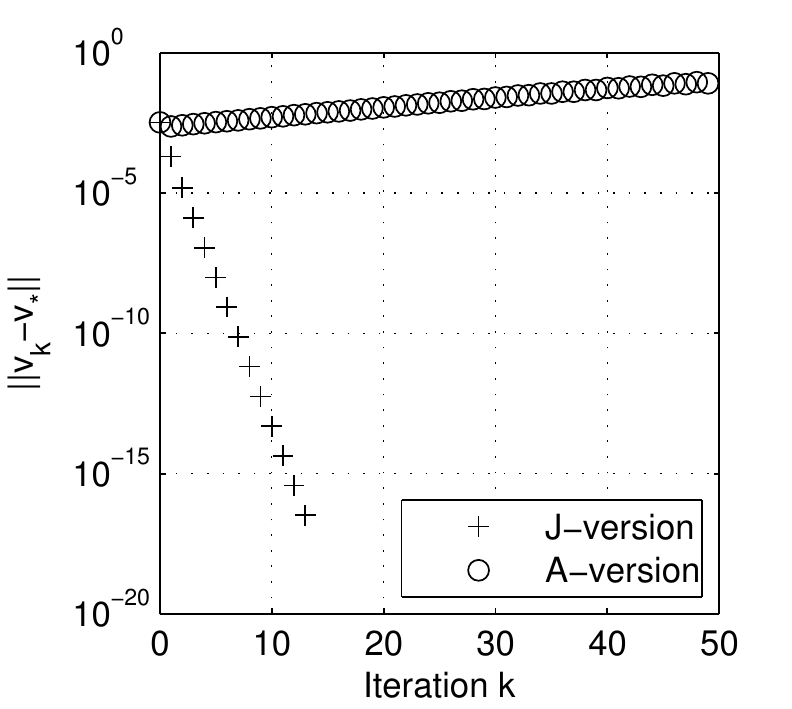}}}\vspace{-0.2cm}
\subfigure[$\beta=1$, $\sigma=\lambda_*+0.3$]{\scalebox{0.6}{\includegraphics{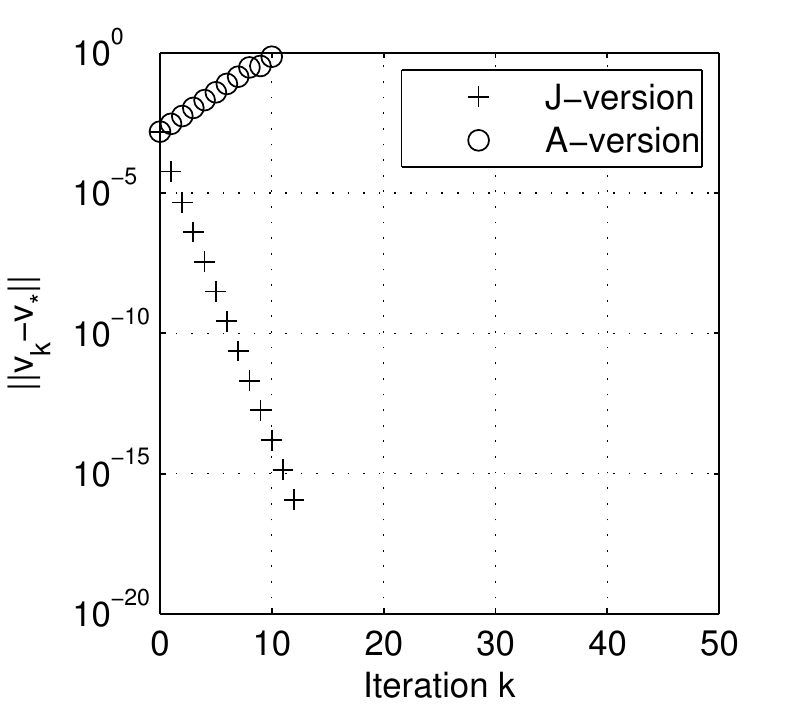}}}\vspace{-0.2cm}%
\subfigure[$\beta=1$, $\sigma= \lambda_*$]{\scalebox{0.6}{\includegraphics{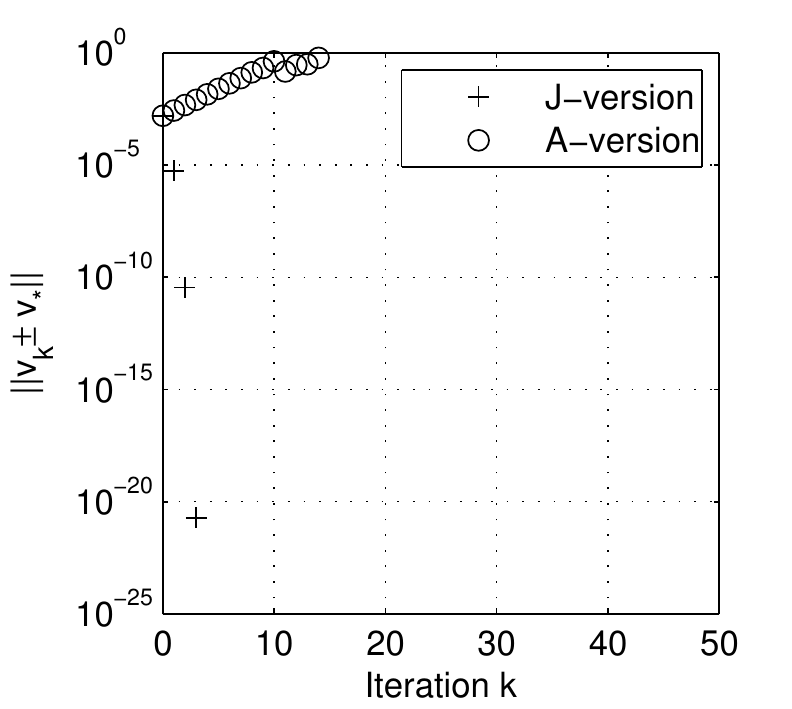}}}\vspace{-0.1cm}%
\end{center}
\caption{
Convergence for the example
in Section~\ref{sect:localconv_exmp}
using \eqref{eq:invit00}, i.e., $J$-version of inverse iteration, 
and the $A$-version of inverse iteration
\eqref{eq:phidef-A}.  \label{fig:AvsJerror}
}
\end{figure}

In Figure~\ref{fig:sinexmp:convfact}
we see that the convergence factor
of the $J$-version approaches zero
when $\sigma\rightarrow\lambda_*$
as predicted by Corollary~\ref{thm:convfact}. The 
convergence factor
for the $A$-version does
clearly not vanish when $\sigma$ 
is close to $\lambda_*$.

\begin{figure}[h]
\begin{center}
\subfigure[]{\scalebox{0.8}{\includegraphics{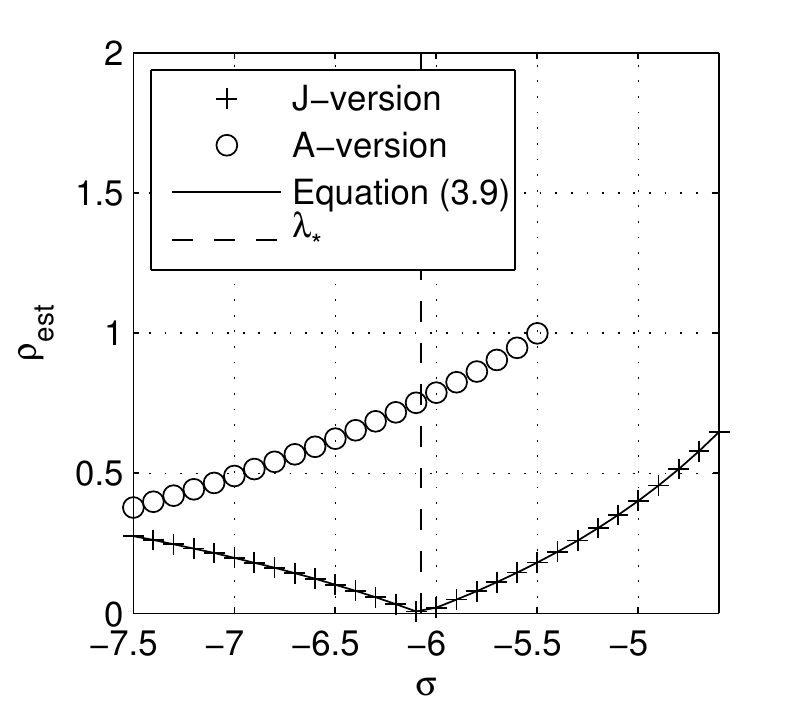}}}%
\subfigure[]{\scalebox{0.8}{\includegraphics{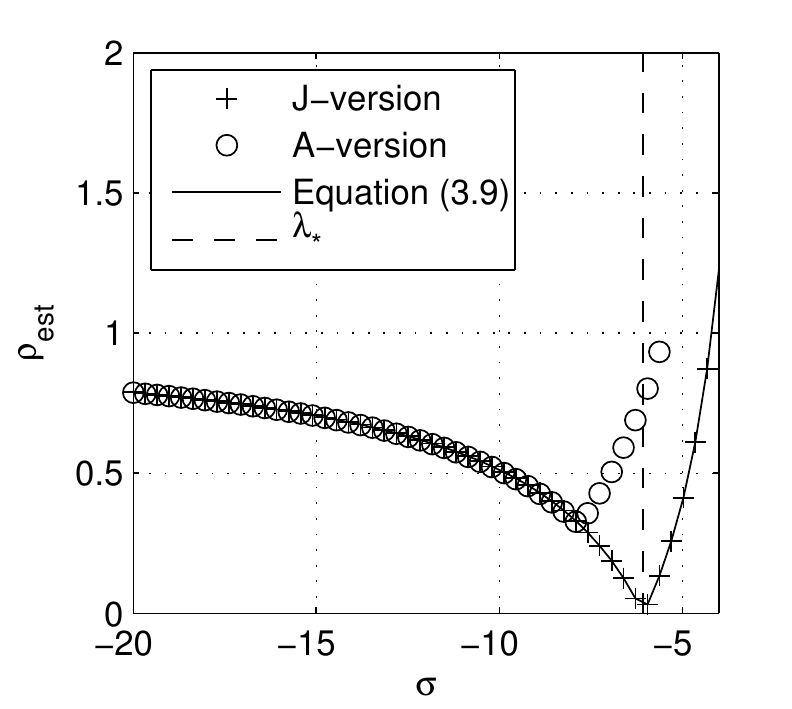}}}
\end{center}
\caption{
Estimated convergence factor for $J$-version
and $A$-version of iteration, where estimation
is done with the quotient $\|v_{k+1}\pm v_K\|/\|v_k\pm v_K\|$
where $v_K$ is a very accurate solution.  Clearly, 
when $\sigma\rightarrow\lambda_*$ 
the convergence factor for the $J$-version
approaches zero, unlike the $A$-version. 
When $\sigma\rightarrow-\infty$ 
the convergence factors coincide. $\beta=0.5$. 
}
\label{fig:sinexmp:convfact}
\end{figure}

\section{Interpretation as the discretization of an ODE}\label{sect:flow}
\subsection{The normalized ODE associated with {(\ref{eq:nepv})}}

In order to provide further insight into 
the non-local behavior of
the iteration, we will  
consider the following ODE
\begin{equation}
 z'(t)=-A(z(t))z(t)\label{eq:flow0}
\end{equation}
and the function corresponding to the normalization of the  solution  
\begin{equation}
  y(t):=q(t)z(t) \textrm{ where }
  q(t):=\frac{1}{\|z(t)\|}.
\label{eq:ydef}
\end{equation}
The function $y$ 
also satisfies an ODE, which does not involve $z$. 
This can be seen from the following reasoning. 
Note that 
\[
 q'(t)= \frac{-1}{2(z(t)^Tz(t))^{3/2}}2z(t)^Tz'(t)=\frac{1}{\|z(t)\|^3}z(t)^TA(z(t))z(t)=q(t)y(t)^TA(y(t))y(t).
\]
By applying the product rule to \eqref{eq:ydef}, we have
\[
  y'(t)=q(t)z'(t)+q'(t)z(t)=\\
-q(t)A(z(t))z(t)+q(t)(y(t)^TA(y(t))y(t))z(t).
\]
Hence, the normalized function $y(t)$ satisfies 
the differential equation 
\begin{equation}
  y'(t)=p(y(t))y(t)-A(y(t))y(t),\label{eq:flow}
\end{equation}
where $p(y(t))$ is the \emph{Rayleigh quotient} defined by \eqref{eq:rq}.
The ODE \eqref{eq:flow} will be used
as a tool to understand the iteration \eqref{eq:invit00}.
For this reason several properties of the ODE  and
some of its relations with the nonlinear eigenvalue problem \eqref{eq:nepv}
will be particularly important.

\begin{theorem}[ODE properties]\label{thm:flowprops}
Consider the ODE \eqref{eq:flow} 
with initial condition $y(0)=y_0$ with $\|y_0\|=1$. 
The ODE and its solution have the following properties. 
\begin{itemize}
\item[(a)] The norm $\|y(t)\|=1$  independent of $t$.
\item[(b)] Any stationary point of the ODE \eqref{eq:flow} is a normalized  eigenvector
of \eqref{eq:nepv}.
\item[(c)] Any  normalized eigenvector  of \eqref{eq:nepv} 
is a stationary point of the ODE \eqref{eq:flow}.
\item[(d)] Let $y_*$ be a stationary point of the ODE \eqref{eq:flow}. 
The stationary point is 
stable if the eigenvalue $\lambda_*=p(y_*)$ 
is a simple dominant eigenvalue of $J(y_*)$, i.e., 
if $\lambda_*$ is simple eigenvalue of $J(y_*)$ and 
\[
  \lambda_*<\re(\mu)\,\,\textrm{ for any }\mu\in\lambda(J(y_*))\backslash\{\lambda_*\}.
\]
\item[(e)] Let $y_*$ be a stationary point of the ODE \eqref{eq:flow}. 
The stationary point is unstable if there is an
eigenvalue $\mu$ of $J(y_*)$ 
such that 
\[
  \re(\mu)< \lambda_*=p(y_*).
\] 
\end{itemize}
\end{theorem}
\begin{proof}
The statements (a)-(c) follow directly from the derivation of \eqref{eq:flow}. 
In order to study the stability of
the stationary point $y_*$, 
let $\Delta(t):=y(t)-y_*$ 
denote the deviation from 
the stationary solution.
Then, by Taylor expansion of \eqref{eq:flow}
we have
\[
  \Delta'(t)=y'(t)=p(y(t))y(t)-A(y(t))y(t)=
\left[p(y_*)I+y_*p'(y_*)-J(y_*)\right]\Delta(t) +O(\Delta(t))^2.
\]
We can now insert $p(y_*)=\lambda_*$ and 
$p'(y_*)=y_*^T(J(y_*)-\lambda I)$ since  $\|y_*\|=\|y_0\|=1$.
Hence, 
 the behavior is to first order given by
\begin{equation}
 \tilde{\Delta}'(t)=F(y_*) \tilde{\Delta}(t),\label{eq:Deltap}
\end{equation}
where
\[
  F(y_*)=  \lambda_* I+y_*y_*^T (J(y_*)-\lambda I)
-J(y_*)=
(I-y_*y_*^T)(\lambda_* I-J(y_*))
\]
The eigenvalues of $F(y_*)$ have the form
\[
  \lambda_*-\mu
\]
where $\mu$ is an eigenvalue of $J(y_*)$, but not equal 
to $\lambda_*$. 
The ODE \eqref{eq:Deltap} is unstable
if $F(y_*)$ has eigenvalues with positive real part,
and the ODE \eqref{eq:Deltap} is stable
if all eigenvalues have negative real part. 
The 
statements (d)-(e) follow  from the sign of 
$\re(\lambda_*-\mu)=\lambda_*-\re(\mu)$, which,
provided it is not zero, gives a conclusion
about the local stability of \eqref{eq:Deltap} and \eqref{eq:flow}.
%
%
%

\end{proof}

\subsection{Discretization of the ODE and equivalence with
inverse iteration}

In the previous subsection we saw how
the ODE \eqref{eq:flow}
was related to the nonlinear 
eigenvalue problem \eqref{eq:nepv}; 
in particular we showed that the stationary solutions were 
equivalent to the eigenvectors of \eqref{eq:nepv}.
We will now show how this connection can
be used by showing that  the 
proposed version of inverse 
iteration \eqref{eq:invit00}
can be interpreted as a discretization
of the ODE, allowing
us to study the general behavior of the iteration by
studying the ODE.

Consider first the Rosenbrock-Euler method 
\cite[Chapter IV.7]{Hairer:1996:ODE}, i.e.,
the backward Euler method where the linear
system is solved with one step of Newton-Raphsons method. 
The backward Euler method with step-length $h$ 
applied to \eqref{eq:flow} is 
\begin{equation}
   \frac{\tilde{y}_{k+1}-\tilde{y}_k}{h}\approx
p(\tilde{y}_{k+1})\tilde{y}_{k+1}-A(\tilde{y}_{k+1})\tilde{y}_{k+1}.\label{eq:beuler}
\end{equation}
%
We approximate the first term in \eqref{eq:beuler} by its linearization, 
\[
  p(\tilde{y}_{k+1})\tilde{y}_{k+1}\approx p(\tilde{y}_k)\tilde{y}_k+(p'(\tilde{y}_k)\tilde{y}_k+p(\tilde{y}_k))(\tilde{y}_{k+1}-\tilde{y}_k)=
             p(\tilde{y}_k)\tilde{y}_{k+1},
\]
where  we denoted $p'(z):=\frac{\partial}{\partial z}p(z)\in\RR^{1\times n}$.
We also used that  $p'(z)z=0$ for any $z$ which follows
from Lemma~\ref{thm:scalinv} since $p(z)$ is scaling invariant.
The second term in \eqref{eq:beuler} is now also approximated
by its linearization. By 
combining the approximation with \eqref{eq:JAidentity}, we have 
\[
  A(\tilde{y}_{k+1})\tilde{y}_{k+1}\approx A(\tilde{y}_k)\tilde{y}_k+
J(\tilde{y}_k)(\tilde{y}_{k+1}-\tilde{y}_k)
=J(\tilde{y}_k)\tilde{y}_{k+1}.
\]
Hence, the Rosenbrock-Euler method applied
to \eqref{eq:flow} 
is equivalent to 
\[
\frac{1}{h}(\tilde{y}_{k+1}-\tilde{y}_k)=p(\tilde{y}_k)\tilde{y}_{k+1}
-J(\tilde{y}_k)\tilde{y}_{k+1}
\]
and also
\begin{equation}
\left(\left(\frac{1}{h}-p(\tilde{y}_k)\right)I 
+J(\tilde{y}_k)
\right)\tilde{y}_{k+1}=
\frac{1}{h}\tilde{y}_k.\label{eq:ykdef}
\end{equation}

The ODE \eqref{eq:flow} has, according to 
Theorem~\ref{thm:flowprops},
a solution with constant norm.
Despite this, the
discretization $\tilde{y}_k$ does
not have constant norm, due to the
approximation error introduced in the stepping scheme.
This issue can be addressed by carrying
out a \emph{standard projection}
described, e.g., in \cite[Algorithm IV.4.2]{Hairer:2006:GEOMETRIC},
which is a common procedure to incorporate normalization constraints.
In this case, it reduces to subsequently normalizing the iterate,
i.e., setting $y_{k+1}:=\frac{1}{\|\tilde{y}_{k+1}\|}\tilde{y}_{k+1}$. 
Since  $p$ and $J$ are scaling invariant, we have
$p(\tilde{y}_k)=p(y_k)$ and $J(\tilde{y}_k)=J(y_k)$, 
and we can express one step of the integration
scheme as 
\begin{equation}
   y_{k+1}
=\alpha_k
\left(\left(\frac{1}{h}-p(y_k)\right)I 
+J(y_k)
\right)^{-1}y_k,\label{eq:rosenbrock}
\end{equation}
where $\alpha_k=1/\left\|\left(\left(\frac{1}{h}-p(y_k)\right)I 
+J(y_k)
\right)^{-1}y_k\right\|_2$.

The main result 
in this section is based on the observation 
that \eqref{eq:rosenbrock} is equivalent to the
proposed version of inverse iteration  \eqref{eq:invit00} 
if the shift $\sigma$ and step-length $h$ are
related in a particular way. This connection is summarized
in the following theorem.

\begin{theorem}[ODE discretization equivalence]\label{thm:flowequiv}
Let the sequences $\{y_k\}_{k=0}^\infty$
and $\{v_k\}_{k=0}^\infty$ be
generated in the following way.
\begin{itemize}
\item[(a)] Let $v_k$ be the iterates generated by inverse
iteration with \eqref{eq:invit00} 
with a given shift $\sigma\in\RR$. 
\item[(b)] Let $y_k$ be the
discretization of \eqref{eq:flow} using
\begin{itemize}
\item[$\bullet$]  the Rosenbrock-Euler method,
\item[$\bullet$]  a standard projection step which imposes the normalization, and 
\item[$\bullet$]  the step-length 
\begin{equation}
  h_k=\frac{1}{p(y_k)-\sigma}\label{eq:hkdef}
\end{equation}
where $\sigma$ is the shift used to generate $v_k$. 
\end{itemize}
\end{itemize}
Moreover, suppose they are
started in the same way, i.e., let $y_0=v_{0}$.
Then,
\[
   v_{k}=y_{k}
\]
for all $k\in \mathbb{N}$. 
\end{theorem}

%



%
%
%
We summarize some immediate consequences. 
\begin{itemize}
\item 
The iteration \eqref{eq:invit00}
is an approximation of the trajectory of the ODE \eqref{eq:flow}
if $h$ is chosen sufficiently small, i.e., if 
$\sigma$ is chosen sufficiently negative. 
\item Suppose the
ODE converges to a stationary solution, which indeed is the
case for several examples, and can in particular be proven for
the example in Section~\ref{sect:GP}. In this
situation the stationary 
solution is a solution to the ODE, and the 
iteration will converge to an eigenvector  for sufficiently 
negative $\sigma$. 
\item For the linear case, i.e., $A(v)=B$, 
we can explicitly write
down the solution to \eqref{eq:flow0} using
the Jordan form, and thereby 
study \eqref{eq:flow} which will converge to the dominant
eigenvalue unless $y(0)$ is chosen in a particular way. 
Moreover, Theorem~\ref{thm:flowprops}d-e
shows that if the dominant eigenvalue is simple, it will be the
only stable solution. 
Analogously,
 for nonlinear problems which resemble the linear
case in the sense that the ODE \eqref{eq:flow}  only
has
one stable stationary solution and the ODE is convergent 
(unless started in a particular way), the iteration will converge to
that solution for sufficiently negative $\sigma$.

\end{itemize}

\begin{remark}[ODE interpration for $A$-version]
The $A$-version \eqref{eq:phidef-A} can be 
interpreted as discretization of the ODE \eqref{eq:flow}, but
where $J$ is approximated by $A$ and $h$ is chosen 
according to \eqref{eq:hkdef}. In fact,
when the $A$-version is applied to the example
in Section~\ref{sect:GP} it is equivalent
to an available in the literature;
more precisely the discretization
of the normalized gradient flow given in 
\cite[Equation (2.20)-(2.21)]{Bao:2004:BOSEEINSTEIN}.
\end{remark}

\subsection{Illustration and numerical justification of ODE discretization}\label{sect:flowill}

Consider a situation where we do not have any
approximation of any eigenvalue of the example 
in  Section~\ref{sect:localconv_exmp}. 
To characterize this
situation we 
carry out a number of simulations with
random starting vectors for different shifts. 
Such a statistical simulation is shown in 
Figure~\ref{fig:ode_ill}a. We observe
that the for $\sigma<-5$, the
iteration appears to always converge to the dominant solution.
This property can be explained using the developed theory. 
When $\sigma\ll \lambda_*$ we can apply Theorem~\ref{thm:flowequiv}
and when $|\sigma-\lambda_*|\ll 1$ we
can apply the local convergence theorem Theorem~\ref{thm:localconv}.

\begin{figure}[t]
\begin{center}
\subfigure[Statistics with  random starting vectors. 
 ]{\scalebox{0.6}{\includegraphics{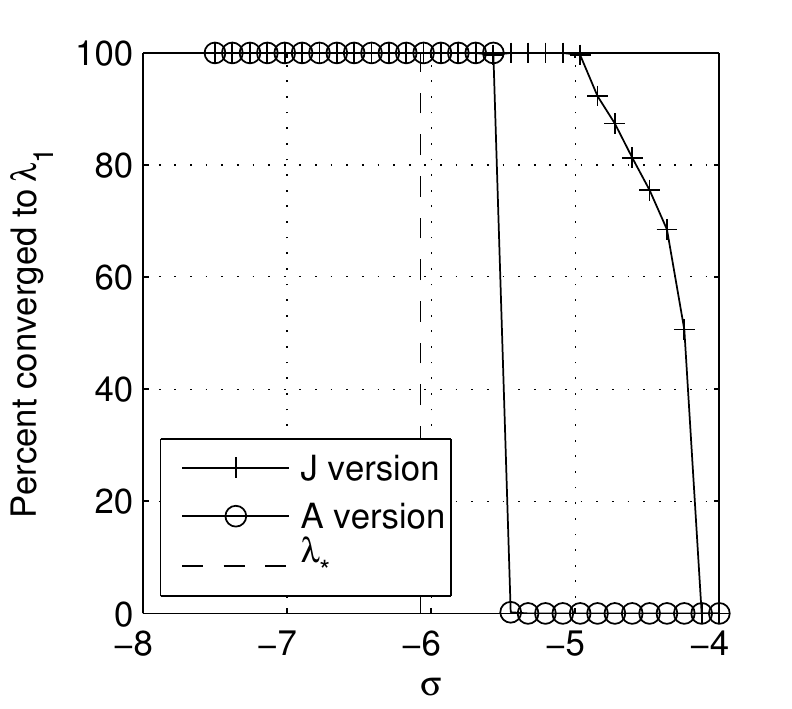}}}%
\subfigure[$\beta=0$]{\scalebox{0.6}{\includegraphics{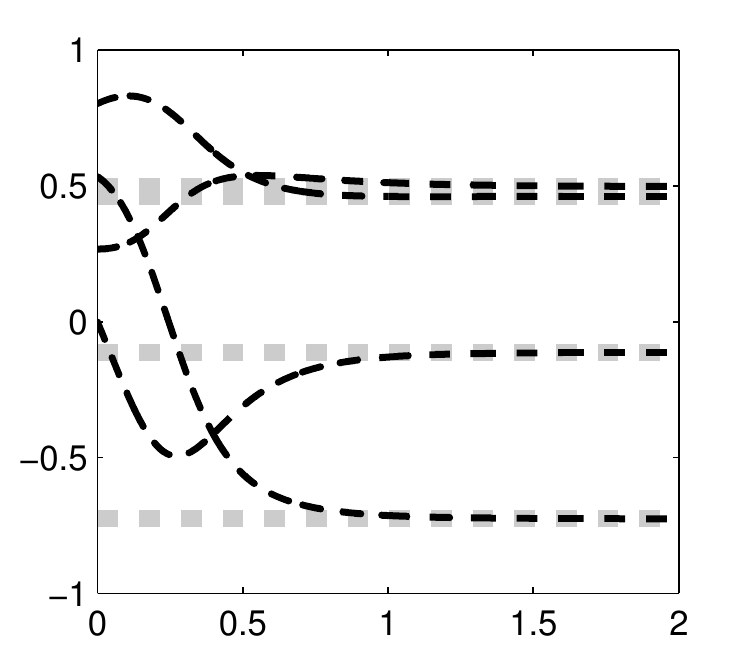}}}%
\subfigure[$\beta=1$]{\scalebox{0.6}{\includegraphics{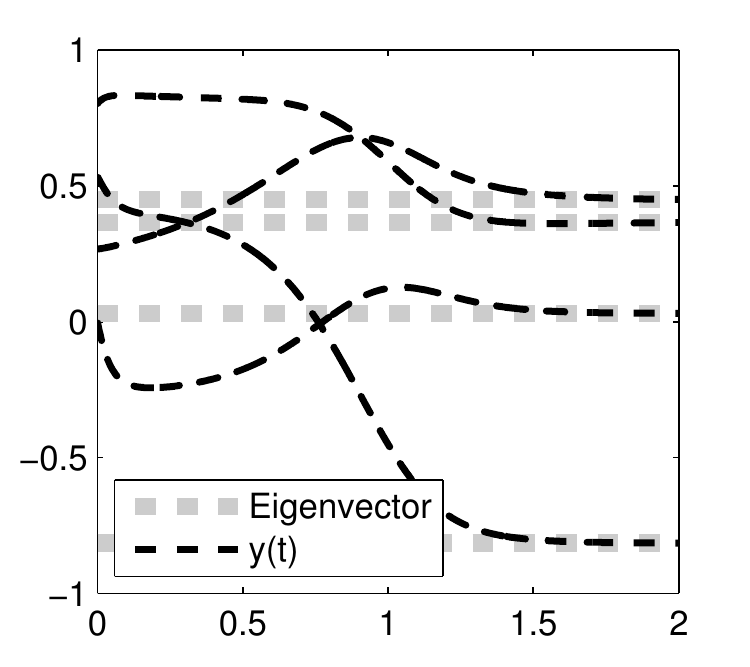}}}
\end{center}
\caption{Convergence of the ODE (sine-example) to the dominant
eigenvector.
\label{fig:ode_ill}
}
\end{figure}

We computed the solution for the ODE \eqref{eq:flow}
for this example using a standard ODE integrator, 
with a random starting value. The trajectory 
is shown in Figure~\ref{fig:ode_ill}b-c. Clearly, 
the ODE converges to the dominant eigenvector. 
Similarly,
in Figure~\ref{fig:ode_ill2} we have shown
the computed approximate trajectories 
using \eqref{eq:invit00} ($J$-version)
and \eqref{eq:phidef-A} ($A$-version). We 
see in accordance with the interpretation
of $\sigma$ as a particular choice of step-length
(via \eqref{eq:hkdef}) that 
the iteration follows the trajectories better if $\sigma$
is more negative. Moreover,
with the  $J$-version 
we can take larger steps than with the $A$-version,
as the $A$-version does not follow
the ODE in Figure~\ref{fig:ode_ill2}a. 

The experience
with the ODE \eqref{eq:invit00}
 for this particular example indicates that
the ODE  converges to a stationary solution for all $\beta\in[0,1]$.  
Moreover, for several choices of $\beta$, among
all eigenvectors, there appears to be only one situation where
$\lambda_*$ is the dominant eigenvalue of $J(v_*)$. This
can be seen in Figure~\ref{fig:ode_ill_J}.
From Theorem~\ref{thm:flowprops}
we consequently know that the ODE \eqref{eq:invit00}
only has one stable stationary solution. 
Combined with the convergence observation, the ODE will always
converge to this solution. Hence, Theorem~\ref{thm:flowequiv}
implies that for sufficiently negative $\sigma$, 
we will follow the trajectory
of the ODE and eventually converge to the dominant eigenvector. 
 \addtolength{\abovecaptionskip}{-0.5cm}
 \addtolength{\belowcaptionskip}{-1.5cm}
 \addtolength{\textfloatsep}{-1.5cm}
\begin{figure}[h]
\begin{center}
\subfigure[$\sigma=-7$]{\scalebox{0.5}{\includegraphics{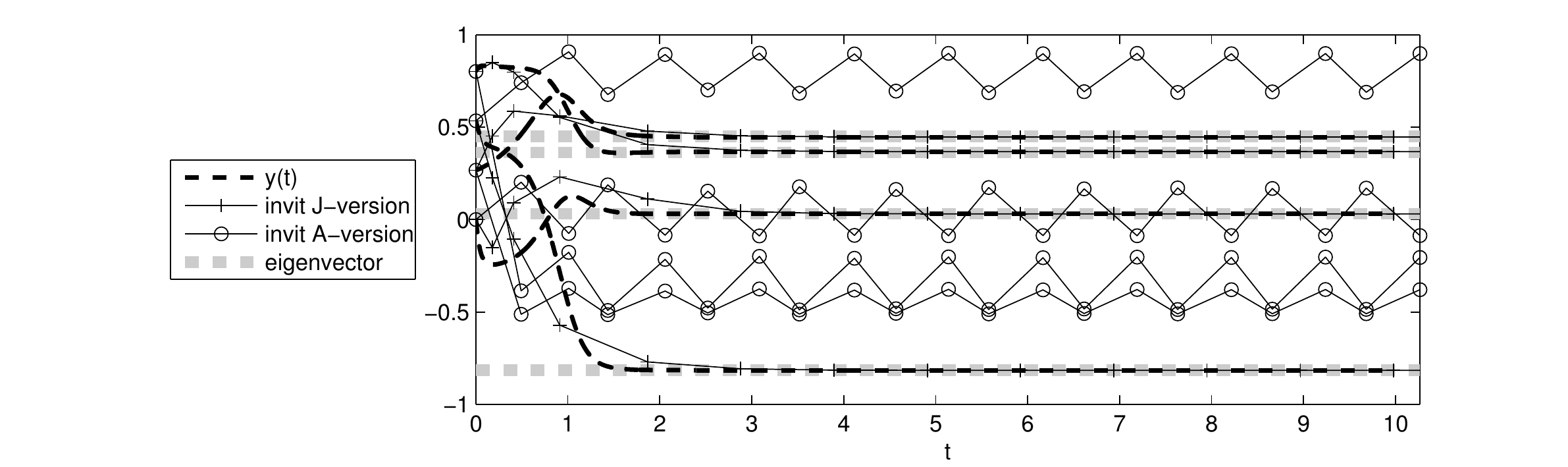}}}\vspace{-0.5cm}
\subfigure[$\sigma=-10$]{\scalebox{0.5}{\includegraphics{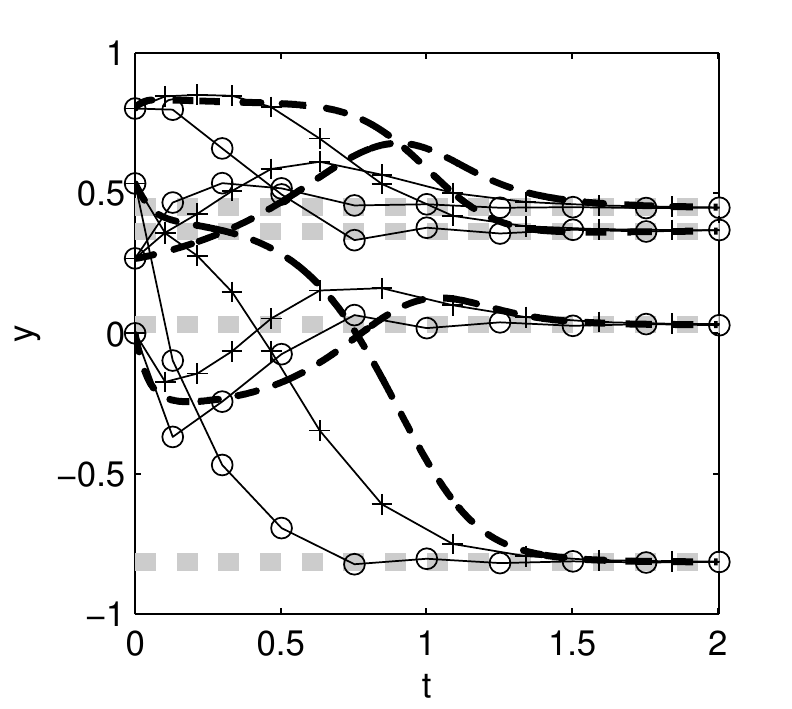}}}%
\subfigure[$\sigma=-20$]{\scalebox{0.5}{\includegraphics{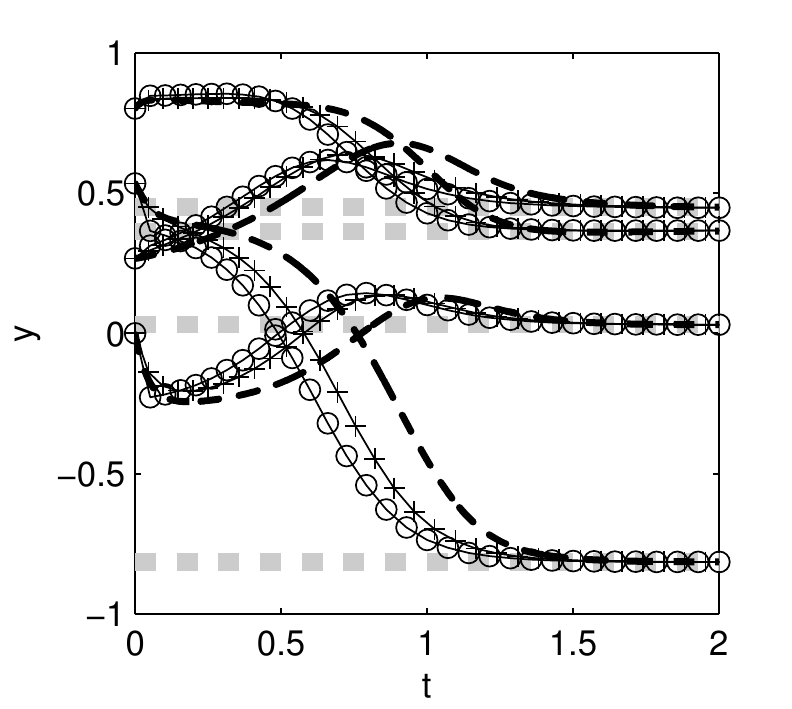}}}%
\subfigure[$\sigma=-50$]{\scalebox{0.5}{\includegraphics{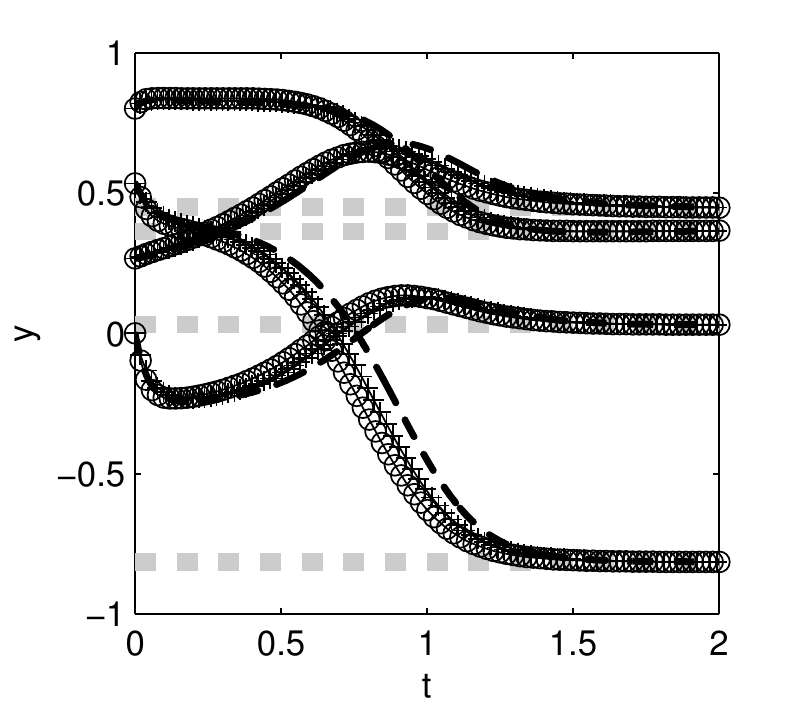}}}%
\end{center}
\caption{$J$-version is
better in the sense that it follows
the ODE also for larger $h$ than for the $A$-version.
The dominant eigenvalue is $\lambda_*\approx -6.01$.
\label{fig:ode_ill2}
}\end{figure}

 \addtolength{\abovecaptionskip}{0.5cm}
 \addtolength{\belowcaptionskip}{1.5cm}
 \addtolength{\textfloatsep}{1.5cm}

\begin{figure}[h]
\begin{center}
\subfigure[$\beta=0$]{\scalebox{0.6}{\includegraphics{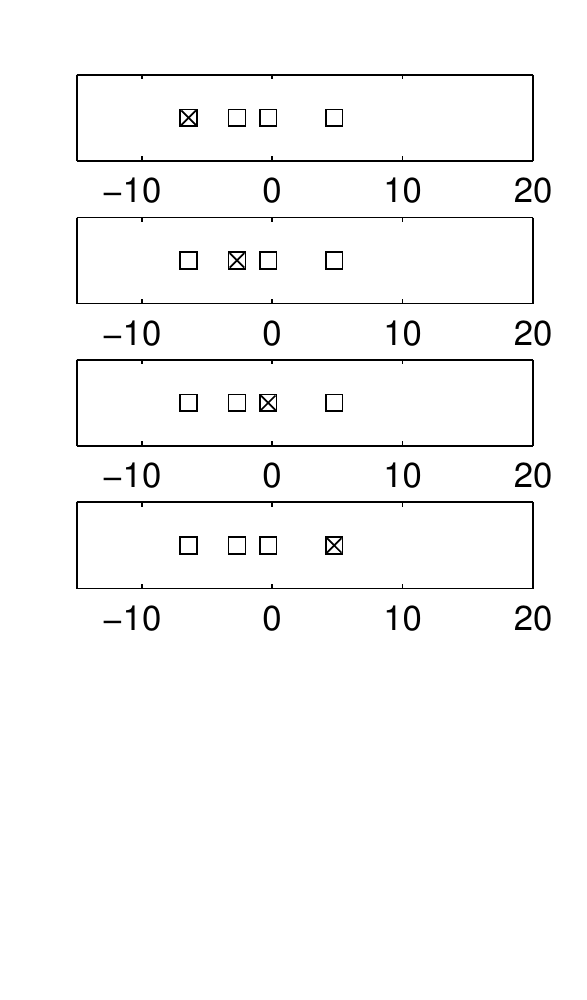}}}
\subfigure[$\beta=0.5$]{\scalebox{0.6}{\includegraphics{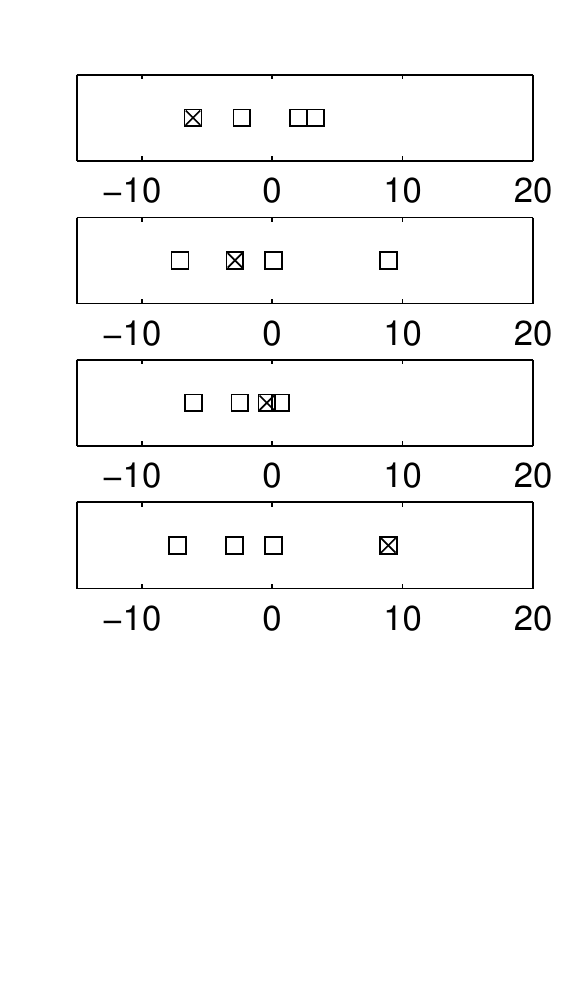}}}
\subfigure[$\beta=1$]{\scalebox{0.6}{\includegraphics{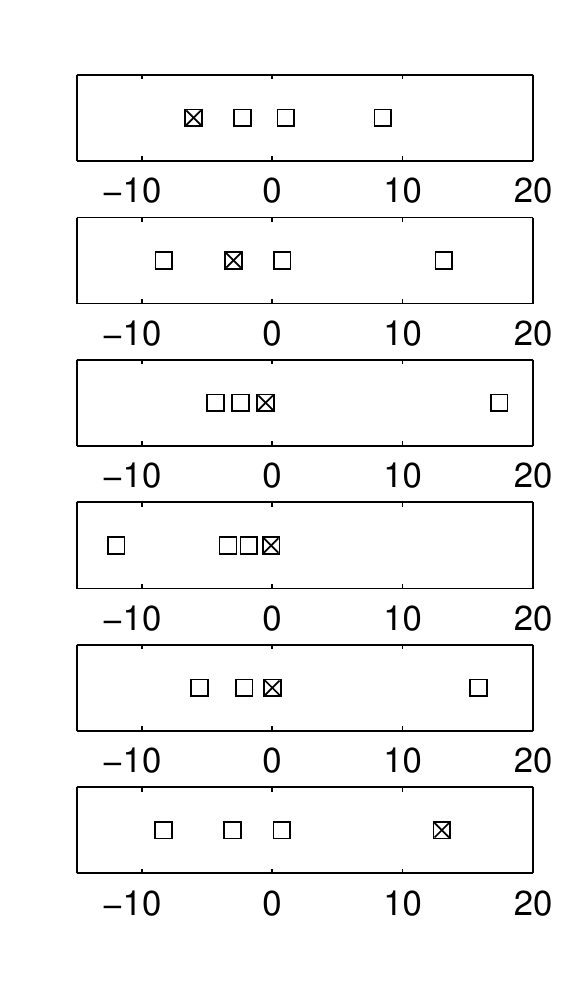}}}
\end{center}
\caption{
The eigenvalues of $J(v_*)$ for the example discussed 
in Section~\ref{sect:flowill}.
In every subfigure, $\square$ denotes the eigenvalues 
of $J(v_*)$ for an eigenpair $(\lambda_*,v_*)$ of \eqref{eq:nepv} and 
$\lambda_*$ is denoted by  $\times$.
In this notation, a stable stationary solution 
described by Theorem~\ref{thm:flowprops}d-e
corresponds to a situation where 
 $\boxtimes$ is to the left of all $\square$.
%
\label{fig:ode_ill_J}
}\end{figure}

\section{Application to the Gross-Pitaevskii equation}\label{sect:GP}

\subsection{Discretization and reformulation to the  form (\ref{eq:nepv})}
To a achieve the physical phenomeon of Bose--Einstein condensation,
weakly interacting particles (such as $^4$He atoms) are trapped and
cooled down to very low temperatures, thereby undergoing a phase
transition into a \emph{Bose--Einstein condensate} (BEC), a superfluid
phase of collective quantum mechanical motion. A standard model for a
BEC is the Gross--Pitaevskii equation (GPE), a nonlinear partial
differential eigenvalue equation. This equation exhibits interesting
topological behavior. In particular, it reproduces experimentally
observed quantized vortices in a BEC within a rotating trap, the
hallmark of superfluidity.

Upon discretization, the GPE becomes an eigenvalue problem of the form
(\ref{eq:nepv}), and to illustrate the iteration (\ref{eq:invit00}) we will
consider the case of a rotating BEC on the domain $\DD = (-L,L)\times
(-L,L)$ in two dimensions, whose GPE reads
\begin{equation}
  \left(-\frac{1}{2}\Delta  - i \Omega \pder{}{\phi} +
    V(x,y)\right)\psi(x,y) + b 
  |\psi(x,y)|^2 \psi(x,y) = \lambda \psi(x,y),\quad (x,y)\in\DD, 
  \label{eq:nlevpv-continuous}
\end{equation}
where $\psi(x,y)$ is the condensate wavefunction (not to be confused
with $\psi$ in Eqn.~(\ref{eq:psidef})) such that
$|\psi(x,y)|^2$ is proportional to the \emph{particle density} at
$(x,y)$, and 
where $\pder{}{\phi}$ is the angular derivative given by
\begin{equation}
  \pder{}{\phi} = y \pder{}{x} - x \pder{}{y}.
  \label{eq:angular-diff}
\end{equation}
The function $V(x,y)$ is the external potential function which
describes the shape of the particle trap. Here, we choose an
asymmetric harmonic oscillator potential $V(x,y) = (x^2 + 1.2y^2)/2$.
We are mostly interested in those solutions to \eqref{eq:nlevpv-continuous}
which are physically important, e.g., the ground state.
 
The boundary condition is chosen as $\psi(x,y)=0$ on
$(x,y)\in\partial\DD$.  Moreover, $\psi$ is normalized such that
$\|\psi\|_{L^2(\DD)} = 1$.  The physical constant $b$ controls the
strength of the interactions between the bosons, and $\Omega$ the
angular velocity of the rotation, which are here selected to
$b=200$ and $\Omega=0.85$, respectively.

The domain is discretized with $N+2$ 
equidistant grid points in each
physical direction leading to $n_c=N^2$ 
interior grid points with spacing $\Delta x$ in both directions. 
For completeness we provide the matrices resulting
from the discretization.
We use
 the approximations of $\Delta$ and $\partial/\partial\phi$, 
\[
  L_N= D_{2,N}\kron I + I\kron D_{2,N},\;\;
L_{\phi,N} = \diag(y_1,\ldots,y_N)\kron D_N - D_N\kron \diag(x_1,\ldots,x_N),
\]
where $D_N$ and $D_{2,N}$ are the central 
difference approximation of the derivative and
the second derivative respectively.  
Moreover,  we let 
\[
\tilde{V}=(V(x_1,y_1),\ldots,V(x_N,y_1), V(x_1,y_2),\ldots,V(x_N,y_2),\ldots V(x_N,y_N))\in\RR^{n}
\]
Then, with
\[
 \tilde{A}_0=-\frac12 L_N-i\Omega L_{\phi,N}+
\diag(\tilde{V})
\]
the discretized problem is 
\begin{equation}
   \tilde{A}(z)z=\lambda z\label{eq:GP_complex}
\end{equation}
where 
\begin{equation}
  \tilde{A}(z) = \tilde{A}_0 + 
\beta \diag(|z|)^2,
\end{equation}
and to be consistent with $\|\psi\|_{L^2(\DD)}=1$ we must have
$\psi(x_j, y_k) = (\Delta x)^{-1} z_{N(k-1)+j}$ and $\beta = (\Delta x)^{-2} b$.

The problem \eqref{eq:GP_complex} is a complex problem of dimension $n_c$,
not satisfying the scaling invariance \eqref{eq:invariance}. 
In order to transform this problem to the form \eqref{eq:nepv}
we introduce the following notation 
\begin{eqnarray*}
  A(v):=
\begin{pmatrix}
\re \tilde{A}(\frac{v_1+iv_2}{\sqrt{v_1^Tv_1+v_2^Tv_2}}) &
-\im \tilde{A}(\frac{v_1+iv_2}{\sqrt{v_1^Tv_1+v_2^Tv_2}})\\
\im \tilde{A}(\frac{v_1+iv_2}{\sqrt{v_1^Tv_1+v_2^Tv_2}})&
\re \tilde{A}(\frac{v_1+iv_2}{\sqrt{v_1^Tv_1+v_2^Tv_2}})
\end{pmatrix}=\\
\begin{pmatrix}
\re \tilde{A}_0 & -\im \tilde{A}_0\\
\im \tilde{A}_0 & \re \tilde{A}_0
\end{pmatrix}
+\frac{\beta}{v^Tv}
B(v),\label{eq:GPE_Adef}
\end{eqnarray*}
and
\[
B(v):=
\begin{pmatrix}\diag(v_1)^2+\diag(v_2)^2 & 0 \\ 0 & \diag(v_1)^2+\diag(v_2)^2\end{pmatrix},\;\;
v=\begin{pmatrix}v_1\\v_2\end{pmatrix}.
\]
With this definition of the matrix $A(v)$ we have
transformed $\tilde{A}(z)$ 
by treating the real and imaginary parts of $z$ as separate variables. 
Consequently, the complex problem \eqref{eq:GP_complex} 
of dimension $n_c$ 
is  equivalent to the real problem 
\begin{equation}
 A(v)v=\lambda v,\label{eq:GPE_A}
\end{equation}
of dimension $n=2n_c$ which does satisfy \eqref{eq:invariance}.

\subsection{Jacobian and exploitation of Jacobian structure} 
The resulting matrix $A(v)$ in \eqref{eq:GPE_A} is a sparse
matrix for any $v$ due to the
finite-difference discretization. In order
to carry out \eqref{eq:invit00} 
we need the Jacobian $J(v)$ and
we need to be able to efficiently solve the corresponding
shifted linear system of equations.
It turns out that the
Jacobian is in general a full matrix, making
the direct application of sparse solvers inefficient. 
Fortunately, the linear
system involving the Jacobian can
be decomposed into two linear systems
of equations involving a matrix
which is sparse.
The derivation is based on the 
fact that the Jacobian is
the sum of a sparse matrix and a rank-one matrix.
For such structures the
Sherman-Morrison-Woodbury formula 
\cite[Section~2.1.3]{Golub:2007:MATRIX} is a standard technique. 


\begin{theorem}[The Jacobian associated with the Gross-Pitaevskii equation]\label{thm:GPE}
The Jacobian for the nonlinear
eigenvalue problem \eqref{eq:GPE_A}  
corresponding to the Gross-Pitaevskii equation 
is given by
\begin{multline}
J(v):=\frac{\partial}{\partial v}(A(v)v)
=
\label{eq:GP_jac}
\begin{pmatrix}
\re A_0 & -\im A_0\\
\im A_0 & \re A_0
\end{pmatrix}+\\
\frac{\beta}{v^Tv}
\left[
\begin{pmatrix}
3\diag(v_1)^2+\diag(v_2)^2 & 2\diag(v_1)\diag(v_2)\\
2\diag(v_1)\diag(v_2) & \diag(v_1)^2+3\diag(v_2)^2 
\end{pmatrix}
-\frac{2}{v^Tv}B(v)vv^T
\right],
\end{multline}
where $v^T=(v_1^T,v_2^T)$ with $v_1,v_2\in\RR^{n/2}$. 
Moreover, suppose $v^Tv=1$ and let 
\begin{multline*}
  C:=\\
\begin{pmatrix}
\re A_0 & -\im A_0\\
\im A_0 & \re A_0
\end{pmatrix}+
\beta
\begin{pmatrix}
3\diag(v_1)^2+\diag(v_2)^2 & 2\diag(v_1)\diag(v_2)\\
2\diag(v_1)\diag(v_2) & \diag(v_1)^2+3\diag(v_2)^2 
\end{pmatrix}-\sigma I
\end{multline*}
and suppose $C$ is non-singular.
Then, the solution
to the linear system  $(J(v)-\sigma I)^{-1}v$ 
can be expressed as 
\begin{equation}
\left(
J(v)-\sigma I 
\right)^{-1}v=u_1+u_2,\label{eq:Jsigmainv_GP}
\end{equation}
where 
\[
   u_1=C^{-1}v 
\]
and 
\[
  u_2=\frac{v^Tu_1}{1-v^Tw}  w 
\;\;\textrm{ with }\;\; 
w=2\beta C^{-1} B(v)v.
\]
\end{theorem}
\begin{proof}
The formula for the Jacobian \eqref{eq:GP_jac}
follows from several applications of the product rule.
More precisely,  we have 
\begin{multline*}
\frac{\partial}{\partial v} A(v)v=
\frac{\partial}{\partial v} \left(
\begin{pmatrix}
\re A_0 & -\im A_0\\
\im A_0 & \re A_0
\end{pmatrix}v
+\frac{\beta}{v^Tv}
B(v)v\right)=\\
\begin{pmatrix}
\re A_0 & -\im A_0\\
\im A_0 & \re A_0
\end{pmatrix}+
\frac{\beta}{(v^Tv)^2}
\left(
v^Tv
\left(\frac{\partial}{\partial v} B(v)v\right)-
2B(v)vv^T
\right).
\end{multline*}
The term involving the
Jacobian of $B(v)v$ can now be simplified, 
\begin{multline*}
\frac{\partial}{\partial v} B(v)v=
\frac{\partial}{\partial v} \left(
\diag(v)^2v+
\begin{pmatrix}0 & I \\ I & 0\end{pmatrix}
\diag(v)^2
\begin{pmatrix}0 & I \\ I & 0\end{pmatrix}v
\right)=\\
3\diag(v)^2+
\begin{pmatrix}0 & I \\ I & 0\end{pmatrix}
\diag(v)^2
\begin{pmatrix}0 & I \\ I & 0\end{pmatrix}
+
2\begin{pmatrix}0 & I \\ I & 0\end{pmatrix}
\diag(v)
\begin{pmatrix}0 & I \\ I & 0\end{pmatrix}
\diag(v)
\begin{pmatrix}0 & I \\ I & 0\end{pmatrix}=\\
B(v)+
2\diag(v)^2+
2\begin{pmatrix}0 & I \\ I & 0\end{pmatrix}
\diag(v)
\begin{pmatrix}0 & I \\ I & 0\end{pmatrix}
\diag(v)
\begin{pmatrix}0 & I \\ I & 0\end{pmatrix}.
\end{multline*}
The relation \eqref{eq:GP_jac}
follows from expanding and combining the three terms.

The formula for \eqref{eq:Jsigmainv_GP}
follows from the fact that the last term in 
\eqref{eq:GP_jac} is a rank-one term and 
we can apply the Sherman-Morrison-Woodbury formula 
\cite[Section~2.1.3]{Golub:2007:MATRIX}.
\end{proof}

%
%
\subsection{Specialized ODE interpretation and step-length heuristics}
\label{sec:heuristics}

We will now use the following important
observation. The ODE \eqref{eq:flow} 
is equivalent to the ODE representing a flow
in the technique known as
imaginary time propagation or normalized
gradient flow in, e.g. \cite{Bao:2004:BOSEEINSTEIN}.
In particular, \eqref{eq:flow}
is equivalent to \cite[Equation (2.15)-(2.16)]{Bao:2004:BOSEEINSTEIN}.
This connection allows us to directly
reach conclusions about the ODE \eqref{eq:flow} for the GPE. 
We conclude from \cite[Remark~2.6]{Bao:2004:BOSEEINSTEIN}
that the ODE will converge to a stationary solution. 
It can equivalently be shown
to converge by constructing a Lyapunov function and applying a variant
of Lyapunov's second method.  (See
\cite[Chapter~3]{Hinrichsen:2005:SYSTEMSTHEORY} for Lyapunov's second
method.)
Moreover, the imaginary time propagation technique 
is for physical
reasons known to converge to a physically 
relevant solution (there may be several), 
e.g., the ground state or meta-stable configurations of the system.

%

In light of the equivalence with imaginary time propagation, it is
natural to follow the true flow as closely as possible during
the iteration, while at the same time taking as long time steps as
possible to minimize computation time. Therefore, we propose the
following heuristic for the step length $h$ or equivalently
a choice of $\sigma$.

Suppose that we have computed an approximation $y_k\approx y(t_k)$,
and we now need to determine an appropriate $h$
to compute $y_{k+1}\approx y(t_k+h)$. 
A step-length
choice for the Rosenbrock-Euler method 
is given  in 
Appendix~\ref{sect:localerr}
and in particular formula \eqref{eq:h_choice_f}, 
for a given fixed local error $\veps$. 
%
Note that for our case, i.e., 
the ODE \eqref{eq:flow}, $f(v)=p(v)v-A(v)v$
and 
\[
f'(v)=-(I-vv^T)(J(v)-p(v)I)+vv^T(A(v)-p(v)I)
\]
such that
\[
  f'(v)f(v)=
(I-vv^T)(-J(v)f(v)+p(v)f(v))+
vv^T(A(v)-p(v)I)f(v).
\]
We will now see that this expression  
can be efficiently computed before each step of the iterative method. 
%



This leads us to  the choice of the
shift $\sigma$ we propose to use in this work:
\begin{itemize}
\item Compute $f_k:=f(v_k)=p(v_k)v_k-A(v_k)y_k$
\item Compute $g_k:=J(v_k)f_k$
\item Compute $e_k:=(I - v_k v_k^T)(-e_k + p(v_k)f_k) + v_k v_k^T
  (A(v_k) - p(v_k)I)f_k$
\item Compute $h_k := (2\veps/\|e_k\|)^{1/2}$ using a desired tolerance $\veps$.
\item If $h_k>h_{\rm max}$ set $h_k=h_{\rm max}$.
\item Set $\sigma=p(v_k)-1/h_k$ according to \eqref{eq:hkdef}.
\end{itemize}
The choice to make sure that $h_k\le h_{\rm max}$ is 
to avoid  taking too large steps, for
which the reasoning for step-length above is not supported. 

Note that we do not need form the matrices $I-v_kv_k^T$ or
$J(v_k)$ explicitly when we apply it to the GPE.
It is more efficient to instead
compute the vectors $f_k$, 
$g_k$ and $e_k$ by forming products between vectors
and matrix vector multiplications with
sparse matrices, since, e.g.,  $J(v_k)$ is given as
the sum of a matrix and a rank one matrix in Theorem~\ref{thm:GPE}).


\subsection{Conclusions from computational results}

We carried out the
inverse iteration algorithm
with the heuristic choice of  $\sigma$ 
proposed in Section~\ref{sec:heuristics}, 
for a number of different 
choices of parameters. 
We selected 
$b = 200$, 
$L = 15$, and $\Omega = 0.85$. The number of grid points is $N=300$, i.e., 
the eigenvalue problem \eqref{eq:nepv} is
of size $n=180000$. The step-length heuristic
was chosen with parameter $\veps=2$ 
and $h_{\rm max}=10^4$ (except where otherwise stated).
An initial guess
was  chosen as a random superposition of Gaussians, such that its
length scale is independent of the interior grid size $N$. 
The simulation was completed in $1.2\cdot 10^3$ seconds with 
an implementation of the algorithm in MATLAB running on an Apple
MacBook Pro with a 2.6 GHz Intel i7 quad-core processor.
 

The results of the numerical simulations are presented in 
Figure~\ref{fig:solution-density}-\ref{fig:sigma-lambda}.
In these figures $\psi_k$ denotes the approximate
solution to \eqref{eq:nlevpv-continuous} 
after iteration $k$ iterations, and $K$ denotes
the total number of iterations. 


%

The convergence is visualized in Figure~\ref{fig:convergence}. 
As expected from the ODE interpretation
in Section~\ref{sect:flow} and the fact that the GPE ODE converges,
we eventually reach a solution. 
Moreover, the asymptotic
convergence is fast as the step-length is larger
when the solution is accurate. 
%
The approximations of the solution at different iterates are
given in Figure~\ref{fig:movie}, 
showing  the shape of the function as it
evolves and converges. Clearly, the random initial condition 
turns into a physically 
meaningful approximation after only a few iterations, and the final iterations mostly
change the position of the vortices.

%

We can also observe the
local convergence properties presented in Section~\ref{sect:localconv}
in this application. 
Note first that when $h=h_{\rm max}$, we can see
that $\sigma$ is almost constant. This can be observed
in Figure~\ref{fig:sigma-lambda}.
It is also expected from the fact 
 that the error behaves like
$\|v_k-v_K\| \approx \|y(t_k)-y_*\|\propto C \exp(-at_k)$
in an asymptotic regime. 
This can 
indeed be confirmed
in Figure~\ref{fig:convergence}. The estimate
clearly indicates that $\|v_{k+1}-v_K\|/\|v_k-v_K\|$ should
converge when $t_{k+1}=t_k+h_{\rm max}$.
Hence, with the assumption that $\sigma$ is
approximately constant in the regime where
we take step-length $h_{\rm max}$ 
we have $\sigma\approx \lambda_*-1/h_{\rm max}$.  
We can compute the theoretical convergence
factor for this $\sigma$ 
by using  Corollary~\ref{thm:convfact} 
and computing $\mu_2$ with the Arnoldi 
method (and a matrix-vector product from  Theorem~\ref{thm:GPE}). 
The theoretical convergence factor 
is visualized together with the estimated
convergence factor $\|v_{k+1}-v_K\|/\|v_k-v_K\|$
in Figure~\ref{fig:conv-factor}. The
theoretical convergence
factor is confirmed for two different choices of $h_{\rm max}$.

The heuristic choice of $\sigma$ is visualized in Figure~\ref{fig:conv-factor}.
With the crude estimation of the relation \eqref{eq:hkdef}, $h\approx 1/(\lambda_*-\sigma)$, we
see that the step-length in the beginning is chosen large, in an
intermediate phase it is chosen around the order of magnitude $10$ 
and in the final phase it is chosen larger, and $\sigma$ is again
eventually almost constant.

\begin{figure}
\begin{center}
\scalebox{0.7}{\includegraphics{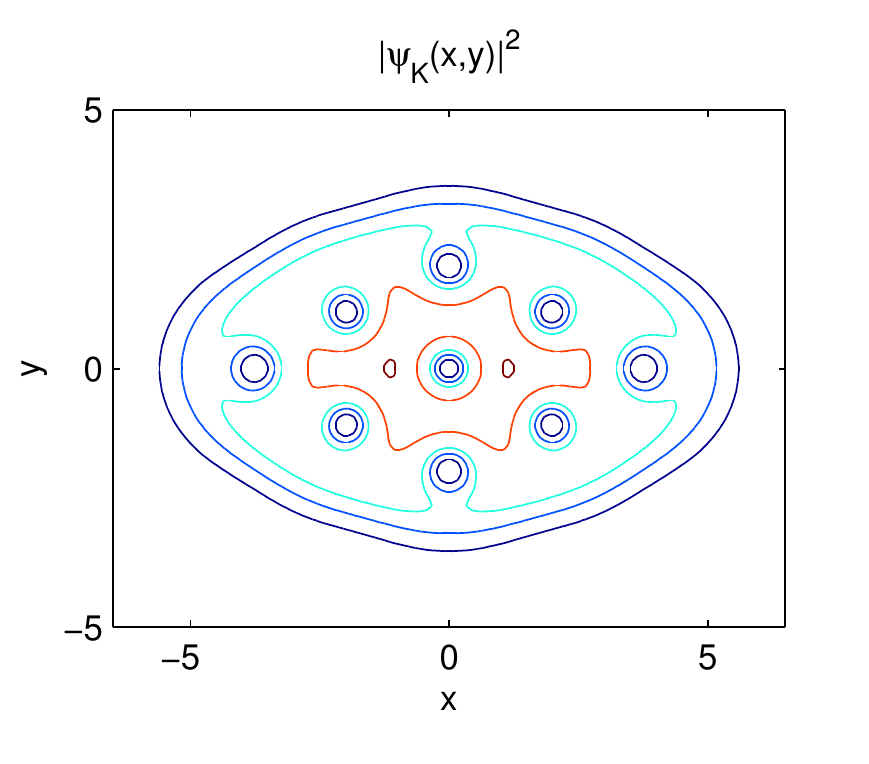}}%
\scalebox{0.7}{\includegraphics{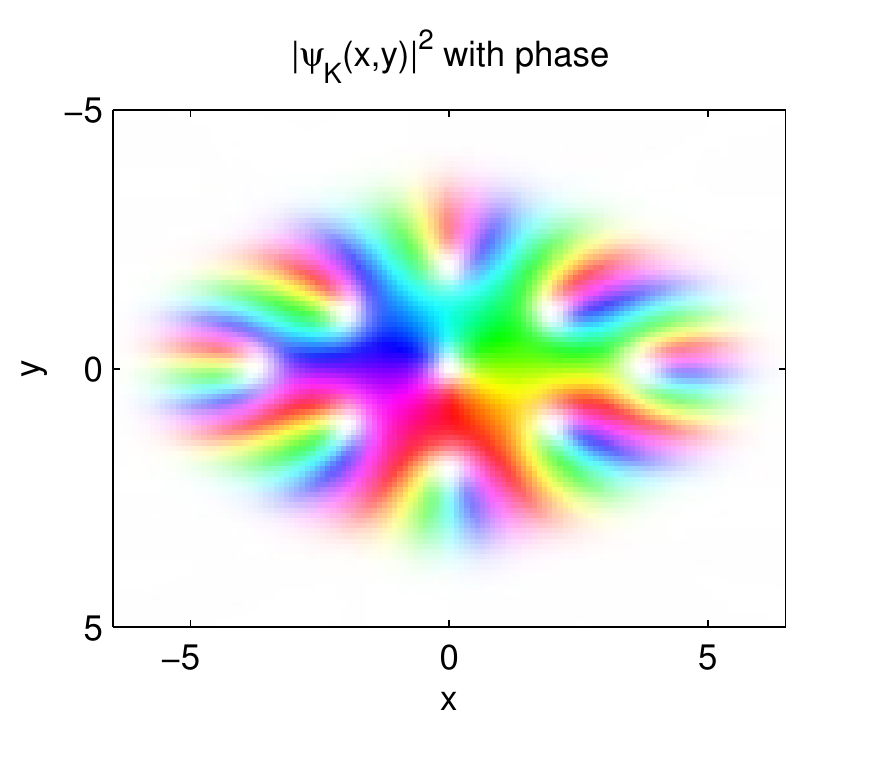}}
\end{center}
\caption{Left: Visualization of the paricle density $|\psi_K(x,y)|^2$ of the
  computed solution. The contour levels are selected as
  $0.05\cdot 10^{-3}$, $0.10\cdot 10^{-3}$, $\cdots$, $0.30\cdot
  10^{-3}$, the outer contours have smaller values. Right:
  Visualization colored according to density and phase. The color
  hue is chosen from the standard color wheel based on the phase
  angle $-i\log (\psi_K/|\psi_K|)$. 
The corresponding eigenvalue approximation
is $\lambda_K = 6.469449$. 
The solution exhibits several vortices arranged in a regular
pattern, as expected for a rotating BEC.
\label{fig:solution-density}}
\end{figure}

\begin{figure}
\begin{center}
\subfigure[]{\scalebox{0.7}{\includegraphics{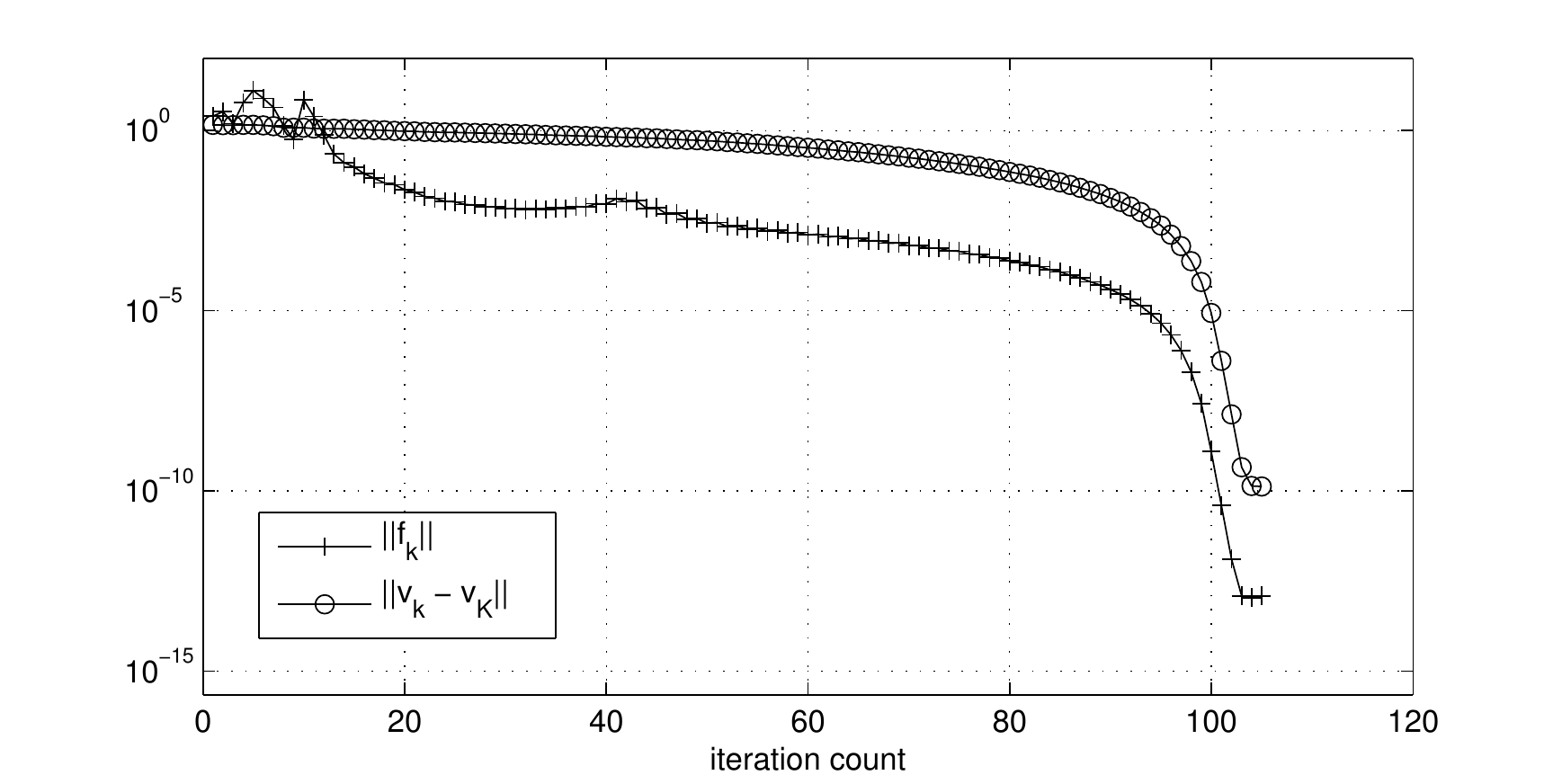}}}
\subfigure[]{\scalebox{0.7}{\includegraphics{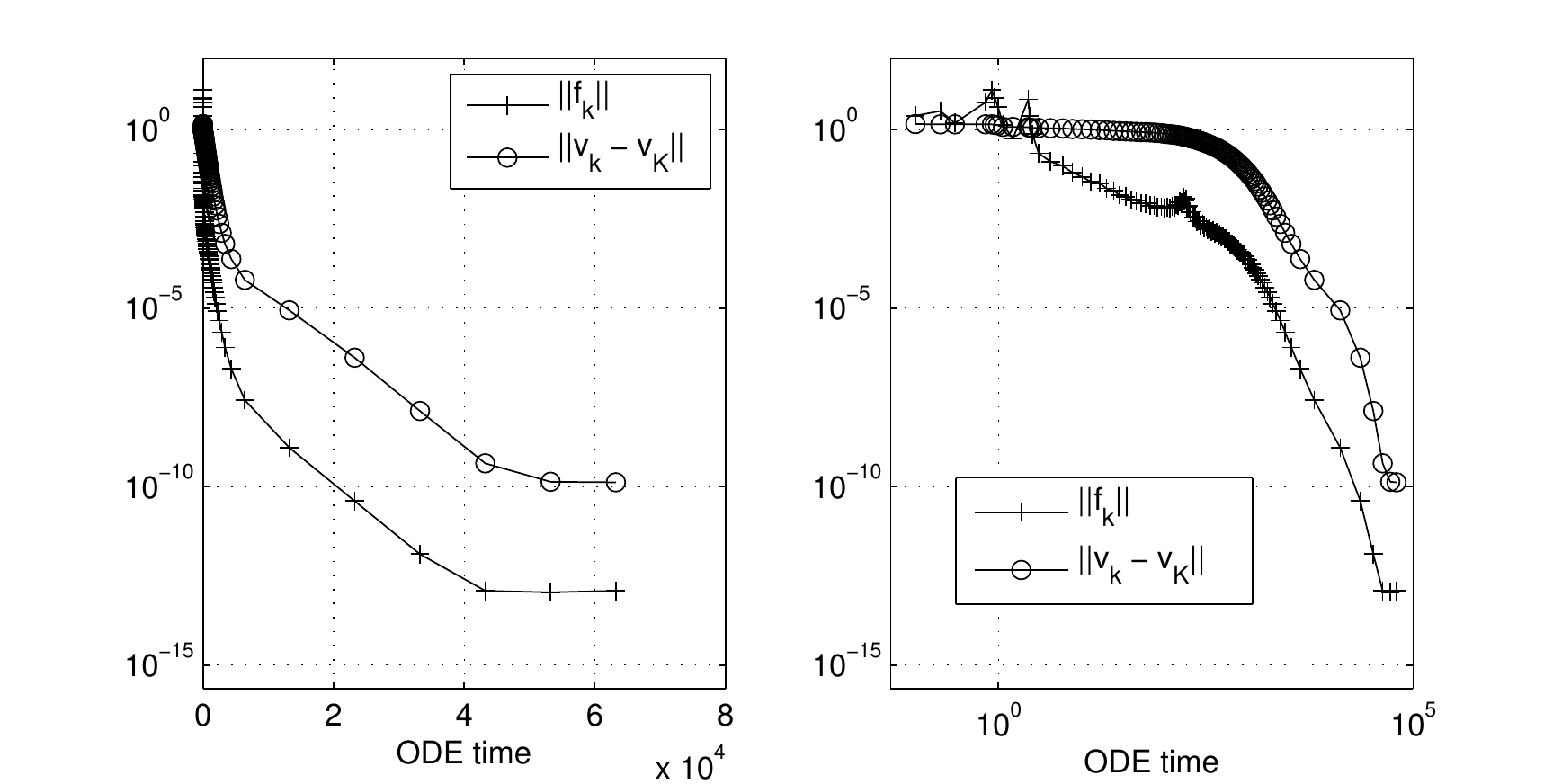}}}
\end{center}
\caption{Plot of norm of residual and absolute error as function of
  iteration count (a) and ODE time $t$ (b). The horizontal axis is
  linear in the left panel, and logarithmic in the right panel. In the final iterations, the time step is $h_\text{max} = 10^4$.
The vector $f_k$ denotes
the residual, i.e.,  $f_k = p(v_k)v_k
- A(v_k)v_k$.
  \label{fig:convergence}}
\end{figure}

\begin{figure}
\begin{center}
\scalebox{0.7}{\includegraphics{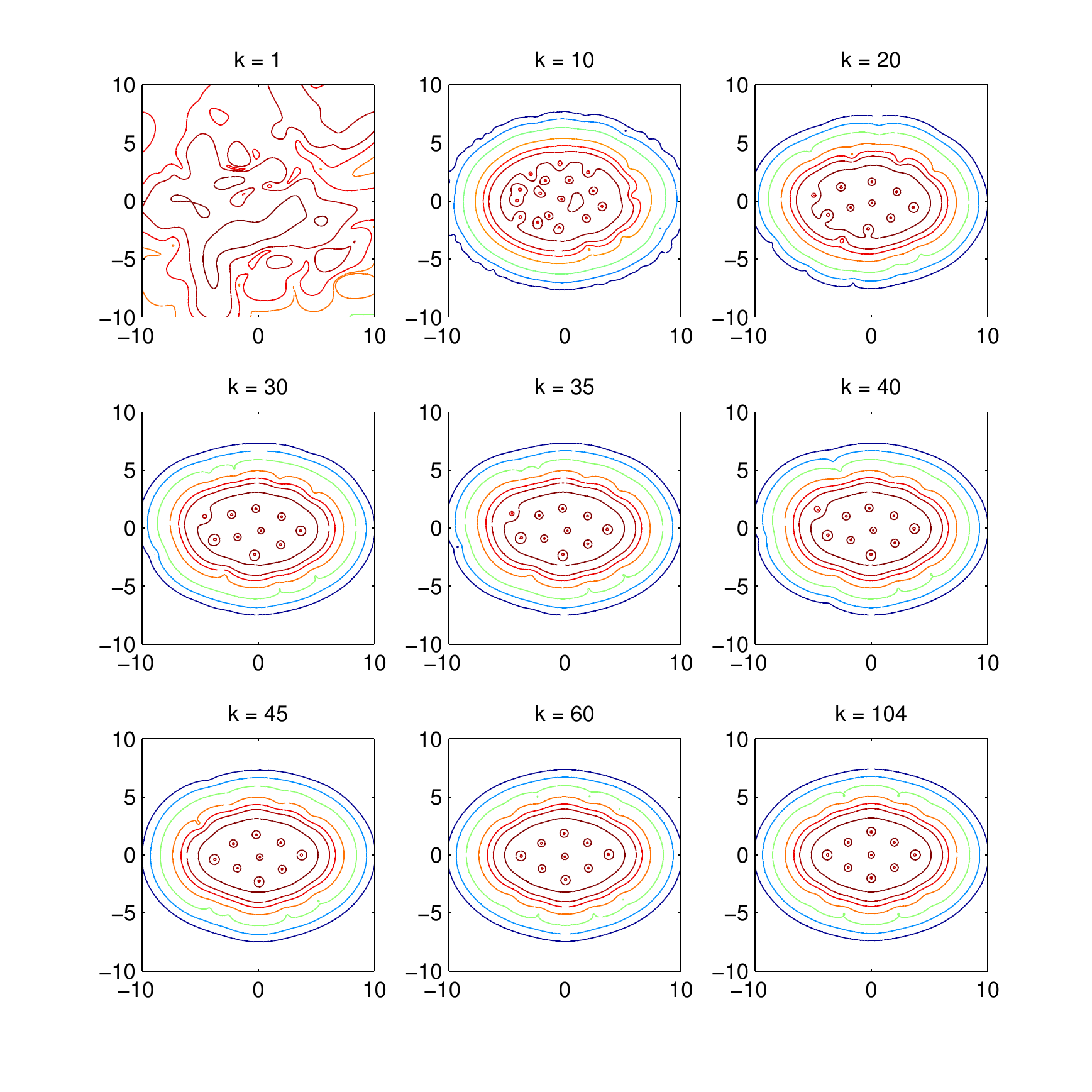}}%
\end{center}
\caption{The approximations $|\psi_k(x,y)|$ for
different iterates. The random starting value
first turns into the general shape of the
solution in $\sim 10$ iterations, and in the remaining iterations,
only the vortices are modified. The
contour levels were selected as  
$10^{-10}$,
$10^{-8}$, 
$10^{-6}$, 
$10^{-4}$, 
$10^{-3}$, 
$10^{-2.5}$, 
$10^{-2}$, 
$10^{-1.75}$, 
$10^{-1.5}$ and 
$10^{-1.25}$. The outer contours have the smaller values.\label{fig:movie}
}
\end{figure}

\begin{figure}
  \begin{center}
    \scalebox{0.7}{\includegraphics{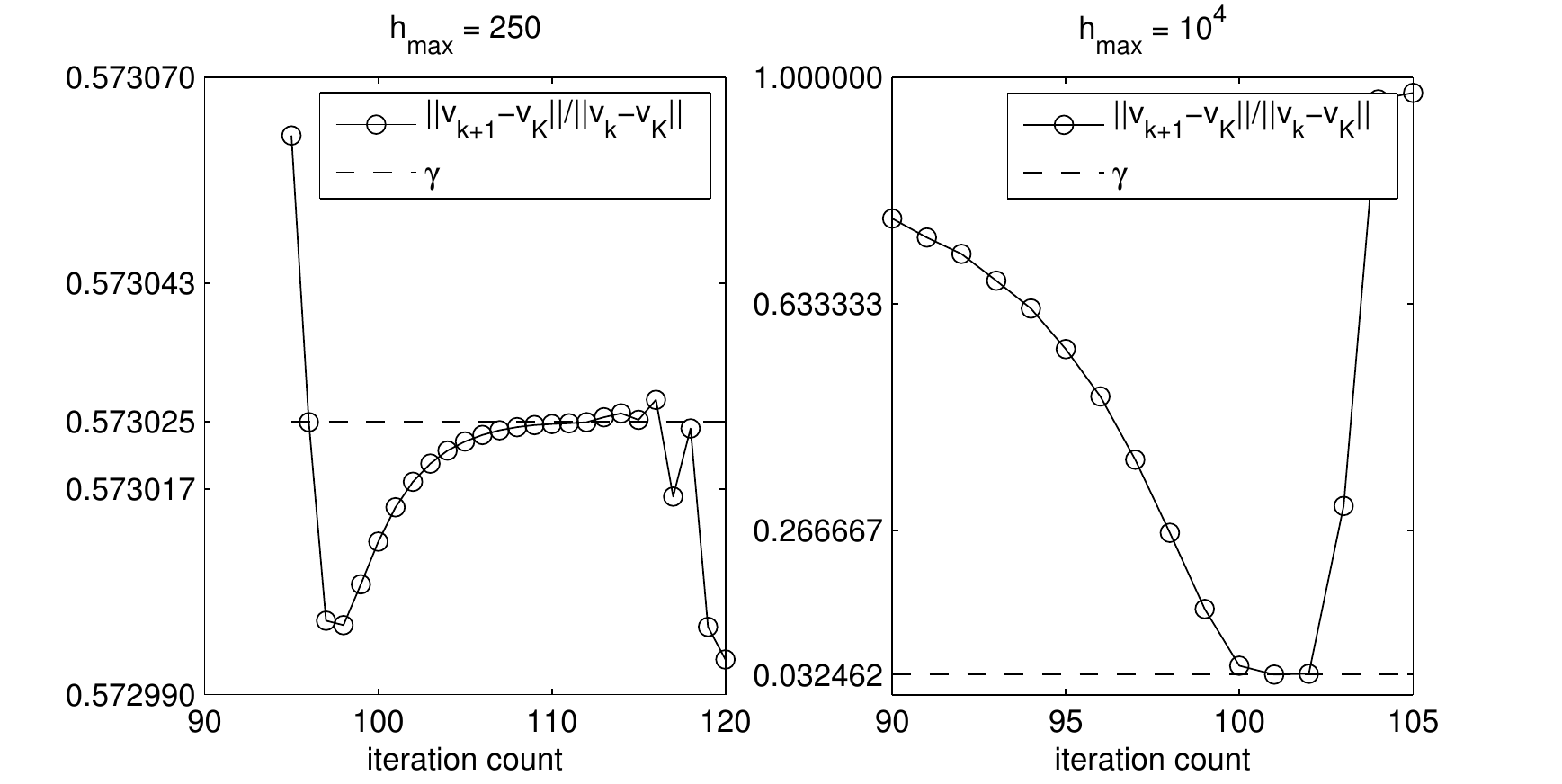}}
  \end{center}
  \caption{Estimation of the  convergence factor
from a calculation with $h_\text{max} =
    250$ and comparison with the theoretical convergence factor
    $\gamma$ (horizontal lines). The asymptotic region starts at
    iteration $k\approx 95$. The deviations for large $k$ is due to
    numerical noise in the error estimate $\|v_k-v_K\|$.\label{fig:conv-factor}}
\end{figure}

\begin{figure}
  \begin{center}
    \scalebox{0.7}{\includegraphics{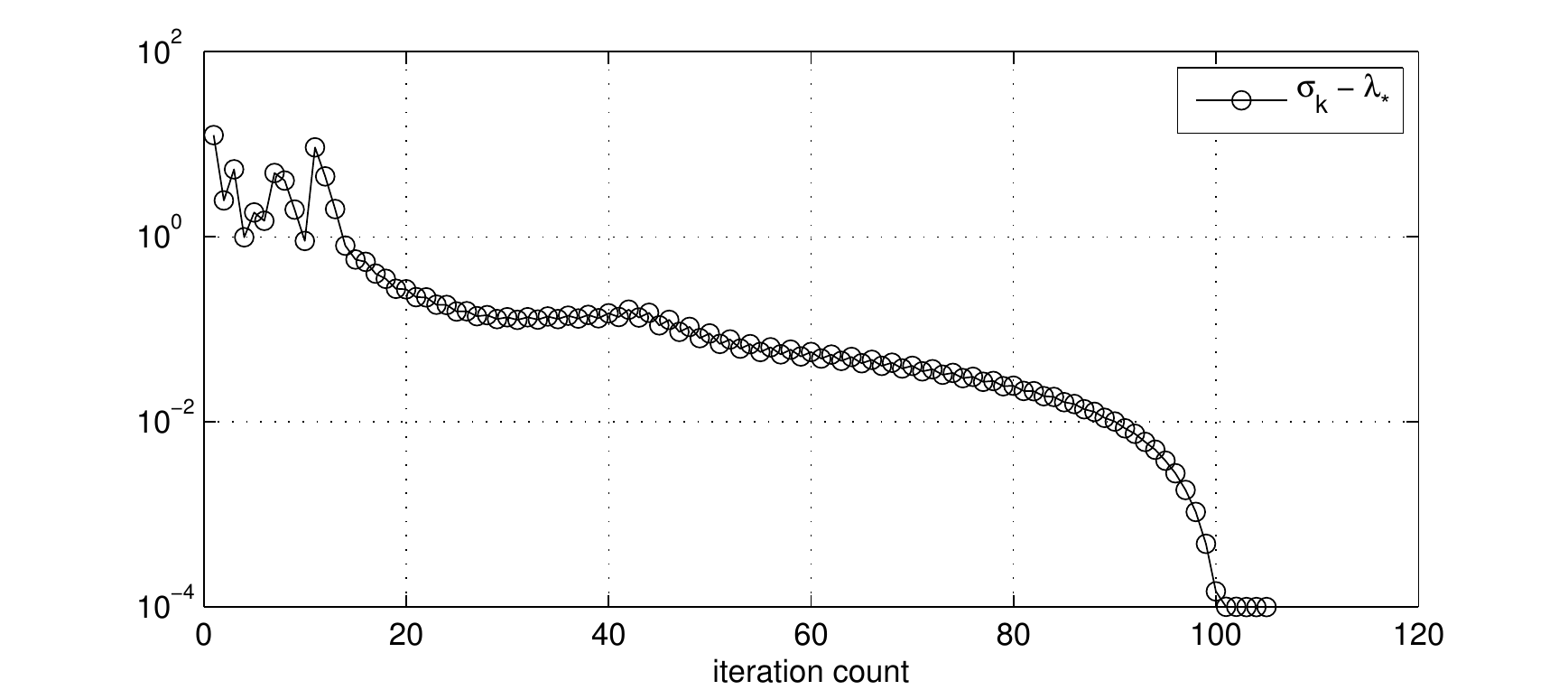}}
  \end{center}
  \caption{Plot of $\sigma_k - \lambda_*$, which quickly becomes small
    and negative.\label{fig:sigma-lambda}}
\end{figure}

%


\section{Concluding remarks} 
The favorable convergence properties
of the  celebrated  inverse iteration algorithm for 
the standard eigenvalue problem 
are well understood. 
An important point in this paper is that 
the generalization we have presented 
have many of the favorable properties
that are present in the inverse iteration
algorithm for standard eigenvalue problems. 
This holds in particular for local convergence
and the interpretation as an ODE. 
We have also illustrated the usefulness of the algorithm
by adapting it to a variant of the Schr\"odinger
equation.

%
%

The connection between the GPE 
and the use of inverse iteration presented
in this paper has
further indirect value.
 For instance,
the tremendous amount of understanding
that is available for inverse iteration
(for  standard eigenvalue problem)
now have the potential to be exploited or adapted 
to this type of nonlinearity. 
%
%

%

{\small \section*{Acknowledgment}

We thank Prof. Tobias Damm (TU Kaiserslautern)
and Prof. Alexander Ostermann (Univ. Innsbruck) 
for providing comments on preliminary 
results related to  this paper. The first 
author gratefully acknowledges the  support of
the Dahlquist research fellowship.
}

\bibliographystyle{plain}
\bibliography{eliasbib,simenbib}

\appendix
\section{Derivation of local error
and step-length 
 for a variant of  Rosenbrock-Euler} \label{sect:localerr}
The error of the Rosenbrock-Euler
method has been studied in the context
of Runge-Kutta methods (e.g. \cite[Chapter~IV.7]{Hairer:1996:ODE}).
We need a more specialized result and will 
for completeness provide an error estimate
for the Rosenbrock-Euler method in our setting. 
Consider the autonomous ODE 
\[
  y'(t)=f(y(t))
\]
with $\|y(0)\|=1$. 
Suppose the ODE has
a structure such that the norm is an invariant, i.e., $\|y(t)\|=1$ for all $t>0$. 
Let $\tilde{y}_1$ be one step of the Rosenbrock-Euler method.
By using Taylor
expansion, it is straightforward to show that 
the local error of the Rosenbrock-Euler step  is given by 
\[
  \tilde{y}_{1}-y(h)= h^2q+\frac{h^3}{6}\hat{y}'''(\tau)
\]
where
\begin{equation}
q:=-\frac{1}{2}(I-hf'(\tilde{y}_0)))^{-1}
(I+hf'(\tilde{y}_0)))f'(\tilde{y}_0)f(\tilde{y}_0).
\label{eq:qdef}
\end{equation}
and $\hat{y}'''(\tau)=(y'''(\tau_1)_1,\ldots,y'''(\tau_n)_n)^T$,
with $\tau_1,\ldots,\tau_n\in [0,h]$. 
We approximate $\|q\|\approx \frac12\|f'(\tilde{y}_0)f(\tilde{y}_0)\|$
and neglect the final term 
which leads us to the following choice of step-length
for a given error tolerance
\begin{equation}
  h=\sqrt{\frac{2\veps}{\|f'(\tilde{y}_0)f(\tilde{y}_0)\|}}.
\label{eq:h_choice_f}
\end{equation}



%
%

%
\end{document}